\documentclass[10pt,a4paper]{amsart}
\usepackage[utf8]{inputenc}
\usepackage[english]{babel}
\usepackage{amsmath}
\usepackage{amsfonts}
\usepackage{amssymb}
\usepackage{amsaddr}
\usepackage{graphicx}
\usepackage{float}
\usepackage{bbm}
\usepackage{fullpage}
\usepackage{hyperref}

\numberwithin{equation}{section}
\newtheorem{proposition}{Proposition}[section]
\newtheorem{thm}[proposition]{Theorem}

\newtheorem{lemma}[proposition]{Lemma}

\begin{document}

\title[Convergence of the random walk loop-soup clusters to CLE]
{Convergence of the two-dimensional random walk loop-soup clusters to CLE}
\author{Titus Lupu}
\address {CNRS and LPSM, UMR 8001,
Sorbonne Université,
4 place Jussieu,
75252 Paris cedex 05,
France}
\email{titus.lupu@upmc.fr}

\begin{abstract}
We consider the random walk loop-soup of sub-critical intensity parameter on the discrete half-plane $\mathtt{H}:=\mathbb{Z}\times\mathbb{N}$. We look at the clusters of discrete loops and show that the scaling limit of the outer boundaries of outermost clusters is a $\hbox{CLE}_{\kappa}$ Conformal loop ensemble.
\end{abstract}

\maketitle

\section{Introduction}
\label{secIntro}

One can naturally associate to a wide class of Markov processes an infinite measure on time-parametrised loops. Roughly speaking, given a locally compact second-countable space $S$, a Markov process $(X_{t})_{0\leq t<\zeta}$ on 
$S$, defined up to a killing time $\zeta\in (0,+\infty]$, 
with transition densities $p_{t}(x,y)$ with respect some $\sigma$-finite
measure $m(dy)$, incorporating the killing if there is one, and with bridge probability measures $\mathbb{P}^{t}_{x,y}(\cdot)$, where the bridges are conditioned on $\zeta >t$, the loop measure associated to $X$ is
\begin{equation}
\label{DefLoopMeasure}
\mu(\cdot)=\int_{x\in S}\int_{t>0}\mathbb{P}^{t}_{x,x}(\cdot)p_{t}(x,x)\dfrac{dt}{t}m(dx).
\end{equation}
See \cite{LeJanMarcusRosen2012Loops} for the precise setting and definition. A \textit{Poisson ensemble of Markov loops} or \textit{loop-soup} of intensity parameter $\alpha >0$ is a Poisson point process of loops of intensity $\alpha\mu$. It is a random countable collection of loops. 
These loop-soups satisfy some universal properties, one of which is the relation to the Gaussian free field at intensity parameter $\alpha=1/2$ \cite{LeJan2011Loops,Lupu2014LoopsGFF}.
We will deal with the clusters of loops. Two loops $\gamma$ and $\gamma'$ in a loop-soup belong to the same cluster if there is a chain of loops $\gamma_{0},\dots,\gamma_{j}$ such that $\gamma_{0}=\gamma$, 
$\gamma_{j}=\gamma'$ and $\gamma_{i}$ and $\gamma_{i-1}$ visit a common point in $S$.

We will consider loop-soups in three different settings.
In the first one, on the continuum half-plane $\mathbb{H}=\lbrace\Im(z)>0\rbrace\subset\mathbb{C}$, we will consider the loop-soups associated to the Brownian motion on $\mathbb{H}$ killed at the first hitting time of the boundary $\mathbb{R}$ and denote them
$\mathcal{L}_{\alpha}^{\mathbb{H}}$. These two-dimensional Brownian loop-soups were introduced by Lawler and Werner in \cite{LawlerWerner2004ConformalLoopSoup} and used by Sheffield and Werner in \cite{SheffieldWerner2012CLE} to give a construction of Conformal loop ensembles ($\hbox{CLE}$).
In \eqref{DefLoopMeasure} we use the same normalisation of the loop measure as in 
\cite{LawlerWerner2004ConformalLoopSoup}, \cite{SheffieldWerner2012CLE}, \cite{LeJan2011Loops} or \cite{LawlerFerreras2007RWLoopSoup}. However, contrary to what is claimed in \cite{SheffieldWerner2012CLE}, the intensity parameter $\alpha$ is not equal to the \textit{central charge} $c$. The central charge is a notion that comes from Conformal Field Theory and representations of Virasoro algebra. Actually,
\begin{displaymath}
\alpha=\dfrac{c}{2}.
\end{displaymath}
The $1/2$ factor was pointed out by Werner in a private communication. It also appears in Lawler's work \cite{Lawler2009PartFuncSLELoopMes}.
The confusion originates from the article
\cite{LawlerWerner2004ConformalLoopSoup}. There the authors consider a Brownian loop
soup in the half-plane and a continuous path cutting the half-plane, parametrised by
the half-plane capacity.  For such a path the half-plane capacity at time
$t$
equals
$2t$
.  It
discovers progressively new Brownian loops and the authors map these loops conformally
to the origin. In Theorem
1
they identify the processes of these conformally mapped
Brownian loops to be a Poisson point process with intensity proportional to the Brownian
bubble measure. In the identification of the intensity there is a factor
2 missing. Actually, in
the article \cite{LawlerWerner2004ConformalLoopSoup}, 
Theorem 1 is inconsistent with Proposition 11.

In the second setting, on the discrete rescaled half-plane 
\begin{displaymath}
\mathtt{H}_{n}:=\left(\dfrac{1}{n}\mathbb{Z}\right)\times\left(\dfrac{1}{n}\mathbb{N}\right),
\end{displaymath}
we will consider the loop-soups associated to the nearest neighbours Markov jump process with uniform transition rates and killed at the first hitting time of the boundary 
$\frac{1}{n}\mathbb{Z}\times\lbrace 0\rbrace$. 
We will denote these loop-soups 
$\mathcal{L}_{\alpha}^{\mathtt{H}_{n}}$. 
The loop-soups associated to Markov jump processes on more general electrical networks were studied by Le Jan in \cite{LeJan2011Loops}. 
If one forgets the parametrisation by continuous time and the "loops" that visit only one vertex, these are exactly the random walk loop-soups studied by Lawler and Trujillo-Ferreras in \cite{LawlerFerreras2007RWLoopSoup}. See also \cite{LawlerLimic2010RW}, Section 9.

In the third setting, we will use the metric (or cable) graphs 
$\widetilde{\mathtt{H}}_{n}$ associated to $\mathtt{H}_{n}$: each 
"discrete" edge $\lbrace (\frac{i}{n},\frac{j}{n}),(\frac{i+1}{n},\frac{j}{n})\rbrace$ or
$\lbrace (\frac{i}{n},\frac{j}{n}),(\frac{i}{n},\frac{j+1}{n})\rbrace$ is replaced by a continuous line of length $\frac{1}{n}$. Let 
$(B^{\widetilde{\mathtt{H}}_{n}}_{t})_{0\leq t<\zeta_{n}}$ be the Brownian motion on
$\widetilde{\mathtt{H}}_{n}$ (cable process) killed at reaching the boundary, that is to say the vertices 
$\frac{1}{n}\mathbb{Z}\times \lbrace 0\rbrace$ 
and all the lines joining $(\frac{i}{n},0)$ to $(\frac{i+1}{n},0)$.
One can find a construction of
$(B^{\widetilde{\mathtt{H}}_{n}}_{t})_{0\leq t<\zeta_{n}}$
in \cite{Lupu2014LoopsGFF}.
Inside each line segment, $B^{\widetilde{\mathtt{H}}_{n}}_{t}$ evolves
like a one-dimensional Brownian motion. After reaching a vertex, the process makes Brownian excursions in each of the four possible directions before hitting the next vertex. Each direction has an equal rate.
$(B^{\widetilde{\mathtt{H}}_{n}}_{2t})_{0\leq t<\zeta_{n}/2}$ converges in law to the Brownian motion on the half-plane $\mathbb{H}$ killed at reaching $\mathbb{R}$. We will denote by $\mathcal{L}_{\alpha}^{\widetilde{\mathtt{H}}_{n}}$ the loop-soups associated to $(B^{\widetilde{\mathtt{H}}_{n}}_{t})_{0\leq t<\zeta_{n}}$. The loop-soups on metric graphs were first considered
in \cite{Lupu2014LoopsGFF}. We will use metric graphs because at intensity parameter $\alpha=1/2$ the probability that two points belong to the same cluster of loops can be explicitly expressed using a metric graph Gaussian free field. Indeed, the clusters of loops are then exactly the sign clusters of the Gaussian free field
\cite{Lupu2014LoopsGFF}.

The discrete loops $\mathcal{L}_{\alpha}^{\mathtt{H}_{n}}$ can be deterministically recovered from the metric graph loops
$\mathcal{L}_{\alpha}^{\widetilde{\mathtt{H}}_{n}}$. The first are the trace on the vertices of the latter. In particular each cluster of $\mathcal{L}_{\alpha}^{\mathtt{H}_{n}}$ is contained in a cluster of $\mathcal{L}_{\alpha}^{\widetilde{\mathtt{H}}_{n}}$, but the clusters of $\mathcal{L}_{\alpha}^{\widetilde{\mathtt{H}}_{n}}$ may be strictly larger \cite{Lupu2014LoopsGFF}.

$c=1$ is the critical central charge for the Brownian loop percolation on $\mathbb{H}$ (or any other simply connected proper subset of $\mathbb{C}$). This means that the critical intensity parameter is $\alpha=1/2$. For $\alpha>1/2$, $\mathcal{L}_{\alpha}^{\mathbb{H}}$ has only one cluster everywhere dense in $\mathbb{H}$. If
$\alpha\in(0,1/2]$, there are infinitely many clusters and each is bounded \cite{SheffieldWerner2012CLE}. 
$\alpha=1/2$ is also the critical intensity parameter for the existence of an unbounded cluster of loops on discrete or metric graph half-plane $\mathtt{H}_{n}$ respectively
$\widetilde{\mathtt{H}}_{n}$ 
\cite{Lupu2014LoopsHalfPlane, Lupu2014LoopsGFF}.
In all three settings, for $\alpha\in(0,1/2]$, we will consider the collection of outer boundaries of outermost clusters (not surrounded by any other cluster) and denote it $\mathcal{F}_{\rm ext}(\mathcal{L}^{S}_{\alpha})$, where $S$ is $\mathbb{H}$, $\mathtt{H}_{n}$ or $\widetilde{\mathtt{H}}_{n}$. Next we give the formal definition of 
$\mathcal{F}_{\rm ext}(\mathcal{L}^{S}_{\alpha})$. We consider the set of all points in $\mathbb{H}$ visited by a loop in
$\mathcal{L}^{S}_{\alpha}$ and take its complement in $\mathbb{H}$. This complement has only one unbounded connected component. We take the boundary in $\mathbb{H}$ of this connected component (by definition it does not intersect $\mathbb{R}$). The elements of $\mathcal{F}_{\rm ext}(\mathcal{L}^{S}_{\alpha})$ are the connected components of this
boundary. We will call the elements of $\mathcal{F}_{\rm ext}(\mathcal{L}^{S}_{\alpha})$
\textit{contours}. The contours are pairwise disjoint and non nested.
See Figure \ref{CEClusters} for a representation of 
$\mathcal{F}_{\rm ext}(\mathcal{L}^{{\widetilde{\mathtt{H}}_{n}}}_{\alpha})$.

\begin{figure}[ht]
\begin{center}
  \includegraphics[scale=0.3]{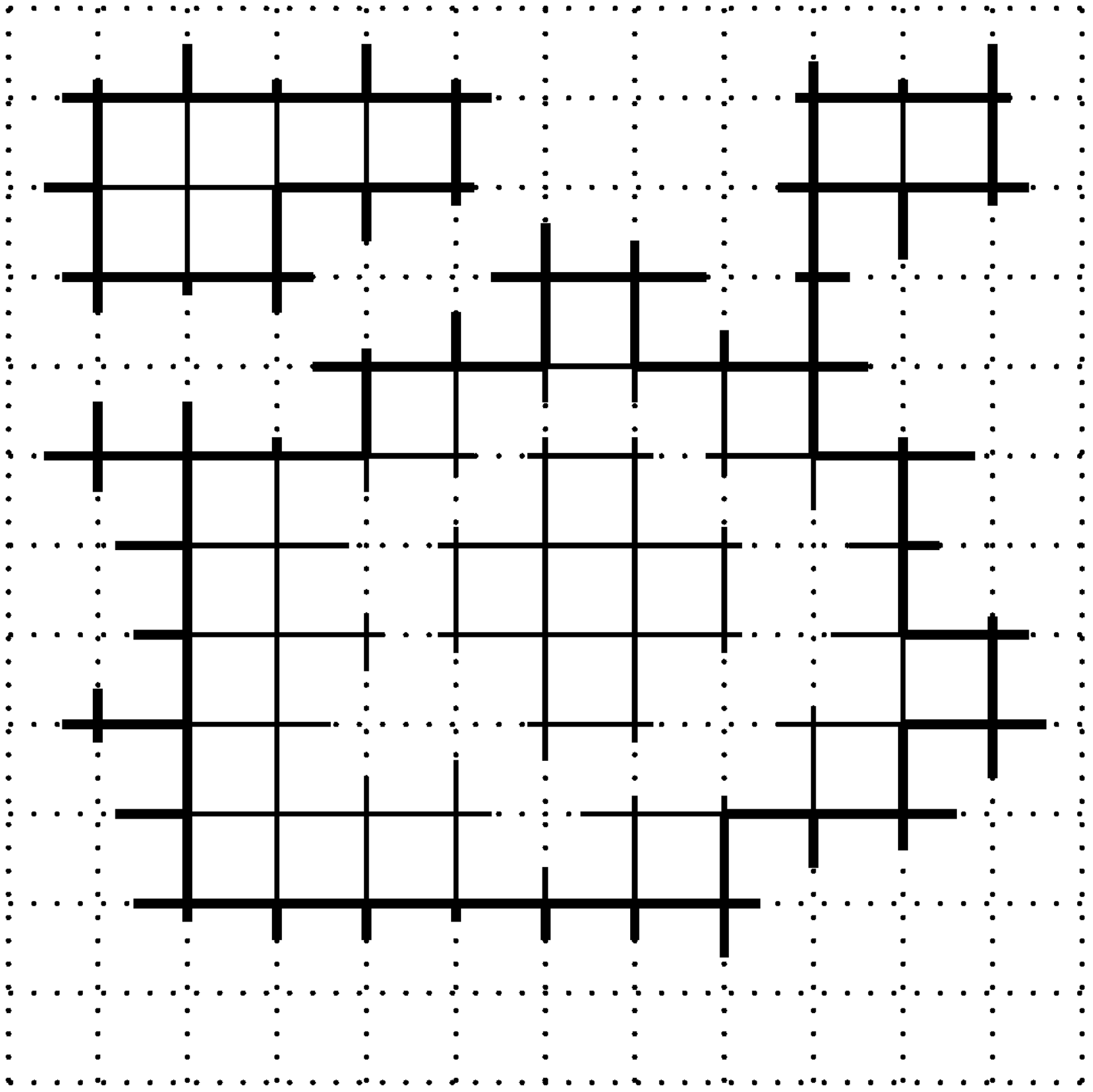}
\end{center}
  \caption{Illustration of three clusters (thin full lines) of 
$\mathcal{L}^{{\widetilde{\mathtt{H}}_{n}}}_{\alpha}$,
two of them being external and one being surrounded.
The thick lines represent the elements of 
$\mathcal{F}_{\rm ext}(\mathcal{L}^{{\widetilde{\mathtt{H}}_{n}}}_{\alpha})$.}
  \label{CEClusters}
\end{figure}

The contours in $\mathcal{F}_{\rm ext}(\mathcal{L}^{\mathbb{H}}_{\alpha})$, $\alpha\in(0,1/2]$, are non self-intersecting loops, and are equal in law to a Conformal loop ensemble $\hbox{CLE}_{\kappa}$,
$\kappa\in(8/3,4]$ \cite{SheffieldWerner2012CLE}.
The relation between $\alpha$ and $\kappa$ is given by
\begin{equation}
\label{ckappa}
2\alpha=c=\dfrac{(3\kappa-8)(6-\kappa)}{2\kappa}.
\end{equation}
We will denote by $\kappa(\alpha)$ the value of $\kappa$ corresponding to a particular intensity parameter $\alpha$.

We will show that both $\mathcal{F}_{\rm ext}(\mathcal{L}^{{\mathtt{H}}_{n}}_{\alpha})$
and $\mathcal{F}_{\rm ext}(\mathcal{L}^{{\widetilde{\mathtt{H}}_{n}}}_{\alpha})$ converge in law to
$\mathcal{F}_{\rm ext}(\mathcal{L}^{\mathbb{H}}_{\alpha})\stackrel{(d)}{=}\hbox{CLE}_{\kappa(\alpha)}$ for
$\alpha\in(0,1/2]$. Observe that $\kappa(1/2)=4$ and 
$\mathcal{F}_{\rm ext}(\mathcal{L}^{{\widetilde{\mathtt{H}}_{n}}}_{1/2})$ and
$\hbox{CLE}_{4}$ are both related to the Gaussian free field. 
$\mathcal{F}_{\rm ext}(\mathcal{L}^{{\widetilde{\mathtt{H}}_{n}}}_{1/2})$ is the collection of outer boundaries of outermost sign clusters of a GFF on the metric graph $\widetilde{\mathtt{H}}_{n}$ \cite{Lupu2014LoopsGFF} and the $\hbox{CLE}_{4}$ loops are in some sense zero level lines of the continuum GFF on $\mathbb{H}$ with zero boundary conditions on
$\mathbb{R}$ \cite{MillerSheffieldCLE4GFF,WangWu2014LvlLineGFF,ASW2017BTLS,
SchrammSheffield2009ContourDiscrGFF,SchrammSheffield2013ContourContGFF}.

Next we define the notion of convergence we will use. $d_{H}$ will be Hausdorff distance on the compact subsets of $\mathbb{H}$. We introduce the distance $d^{\ast}_{H}$ between finite collections of compact subsets of $\mathbb{H}$:
\begin{displaymath}
d^{\ast}_{H}(\mathcal{K},\mathcal{K}')=
\left\lbrace
\begin{array}{ll}
+\infty & \text{if}~\vert\mathcal{K}\vert\neq\vert\mathcal{K}'\vert, \\ 
\min_{\sigma\in \mathrm{Bij}(\mathcal{K},\mathcal{K}')}\max_{K\in \mathcal{K}}d_{H}(K,\sigma(K)) & \text{otherwise},
\end{array}   
\right.
\end{displaymath}
where $\mathcal{K}$ and $\mathcal{K}'$ are finite collections of compact subsets and $\mathrm{Bij}(\mathcal{K},\mathcal{K}')$ is the set of all bijections from $\mathcal{K}$ to $\mathcal{K}'$.
Given $z\in\mathbb{H}$, we will denote by
\begin{displaymath}
\mathcal{F}_{\rm ext}(\mathcal{L}^{S}_{\alpha})(z)
\end{displaymath}
the contour of $\mathcal{F}_{\rm ext}(\mathcal{L}^{S}_{\alpha})$ that contains or surrounds $z$, whenever it exists. It exists a.s.\ in the case $S=\mathbb{H}$. Given $z_{1},\dots,z_{j}\in\mathbb{H}$, we will denote
\begin{displaymath}
\mathcal{F}_{\rm ext}(\mathcal{L}^{S}_{\alpha})[z_{1},\dots,z_{j}]:=
\lbrace \mathcal{F}_{\rm ext}(\mathcal{L}^{S}_{\alpha})(z_{i})\vert 1\leq i\leq j\rbrace.
\end{displaymath}
By the convergence in law of $\mathcal{F}_{\rm ext}(\mathcal{L}^{{\mathtt{H}}_{n}}_{\alpha})$ and
$\mathcal{F}_{\rm ext}(\mathcal{L}^{{\widetilde{\mathtt{H}}_{n}}}_{\alpha})$ to
$\mathcal{F}_{\rm ext}(\mathcal{L}^{\mathbb{H}}_{\alpha})$ we mean that for any $z_{1},\dots,z_{j}\in\mathbb{H}$,
$\mathcal{F}_{\rm ext}(\mathcal{L}^{{\mathtt{H}}_{n}}_{\alpha})[z_{1},\dots,z_{j}]$ and 
$\mathcal{F}_{\rm ext}(\mathcal{L}^{{\widetilde{\mathtt{H}}_{n}}}_{\alpha})[z_{1},\dots,z_{j}]$ converge in law
to $\mathcal{F}_{\rm ext}(\mathcal{L}^{\mathbb{H}}_{\alpha})[z_{1},\dots,z_{j}]$ for the distance
$d^{\ast}_{H}$.

So, the main result in this article is the following.
\begin{thm}
\label{ThmConvAll0} 
Let $\alpha\in (0,1/2]$.
$\mathcal{F}_{\rm ext}(\mathcal{L}_{\alpha}^{\mathtt{H}_{n}})$ and 
$\mathcal{F}_{\rm ext}(\mathcal{L}_{\alpha}^{\widetilde{\mathtt{H}}_{n}})$ converge in law (in the above defined sense) as
$n\rightarrow +\infty$ to $\mathcal{F}_{\rm ext}(\mathcal{L}_{\alpha}^{\mathbb{H}})$, that is to say to a 
$\hbox{CLE}_{\kappa(\alpha)}$ on $\mathbb{H}$.
\end{thm}

In the article \cite{BrugCamiaLis2014RWLoopsCLE} Van de Brug, Camia and Lis consider clusters of rescaled two-dimensional random walk loops that are not too small. Given $T>0$ let $\mathcal{L}_{\alpha}^{\mathtt{H}_{n},T}$ be the
subset of $\mathcal{L}_{\alpha}^{\mathtt{H}_{n}}$ consisting of random walk loops that do at least $T$ jumps.
In \cite{BrugCamiaLis2014RWLoopsCLE} it is almost shown that for 
$\theta\in(16/9,2)$ and $\alpha\in(0,1/2]$, 
$\mathcal{F}_{\rm ext}(\mathcal{L}_{\alpha}^{\mathtt{H}_{n},n^{\theta}})$ converges in law to a $\hbox{CLE}_{\kappa(\alpha)}$ process in the sense described previously. The result uses the approximation of "not too small" Brownian loops by "not too small" random walk loops obtained by Lawler and Trujillo-Ferreras in \cite{LawlerFerreras2007RWLoopSoup}. However the authors in \cite{BrugCamiaLis2014RWLoopsCLE} consider the loop-soups only on bounded domains. 
In the present paper, we will extend their result by removing the cutoff on microscopic loops (and also consider the case of unbounded domains).
Actually, the "microscopic" loops that are thrown away in \cite{BrugCamiaLis2014RWLoopsCLE} create additional connections and may merge
large clusters. So the point is to show that this happens with a probability converging to $0$ and the contribution of microscopic loops does not change the picture at macroscopic level.
Observe that in \cite{BrugCamiaLis2014RWLoopsCLE} the authors use the same normalisation of the measure on loops as we do but with the widespread confusion about the factor $2$ in the intensity of loop-soups.

From above considerations one deduces that the contours obtained in the limit from
$\mathcal{F}_{\rm ext}(\mathcal{L}_{\alpha}^{\mathtt{H}_{n}})$ and \textit{a fortiori} from
$\mathcal{F}_{\rm ext}(\mathcal{L}_{\alpha}^{\widetilde{\mathtt{H}}_{n}})$ are "at least as big as" 
$\hbox{CLE}_{\kappa(\alpha)}$ loops. We thus have a "lower bound". To conclude the convergence we need an "upper bound". We will prove Theorem \ref{ThmConvAll0}
in two steps.
First, we will construct an "upper bound" for 
$\mathcal{F}_{\rm ext}(\mathcal{L}_{1/2}^{\widetilde{\mathtt{H}}_{n}})$ and deduce the convergence to $\hbox{CLE}_{4}$ of 
$\mathcal{F}_{\rm ext}(\mathcal{L}_{1/2}^{\widetilde{\mathtt{H}}_{n}})$ and
$\mathcal{F}_{\rm ext}(\mathcal{L}_{1/2}^{\mathtt{H}_{n}})$. Then from this we will deduce the desired convergences for $\alpha\in(0,1/2)$.
For this, we will divide the loop-soup of intensity $1/2$ in two independent
loop-soups of respective intensities $\alpha$ and $\bar{\alpha}$,
with $\alpha+\bar{\alpha}=1/2$.
If the scaling limit of 
$\mathcal{F}_{\rm ext}(\mathcal{L}_{\alpha}^{\widetilde{\mathtt{H}}_{n}})$
happens to contain contours "strictly larger" than $\hbox{CLE}_{\kappa(\alpha)}$,
then the additional independent contribution of 
$\mathcal{L}_{\bar{\alpha}}^{\widetilde{\mathtt{H}}_{n}}$
would give in the scaling limit of
$\mathcal{F}_{\rm ext}(\mathcal{L}_{1/2}^{\widetilde{\mathtt{H}}_{n}})$ contours
"strictly larger" than $\hbox{CLE}_{4}$, and this would contradict the first step.

Next we explain how the "upper bound" in the critical case
$\alpha=1/2$ will be constructed.
We additionally introduce two Poisson point processes of excursions on 
$\widetilde{\mathtt{H}}_{n}$ and on $\mathbb{H}$. First we consider $\widetilde{\mathtt{H}}_{n}$. 
Let $x\in \frac{1}{n}\mathbb{Z}_{-}\times \lbrace 0\rbrace$, where 
$\mathbb{Z}_{-}$ includes $0$.
Let ${\nu_{\rm exc}^{\widetilde{\mathtt{H}}_{n}}(x\rightarrow (-\infty,0])}$ be the measure on excursions of the metric graph Brownian motion $B^{\widetilde{\mathtt{H}}_{n}}$ from $x$ to a point in 
$\frac{1}{n}\mathbb{Z}_{-}\times \lbrace 0\rbrace$. It is defined as follows: Let 
$\mathbb{P}_{x+i\varepsilon}^{\widetilde{\mathtt{H}}_{n}}(\cdot, 
B^{\widetilde{\mathtt{H}}_{n}}_{\zeta_{n}^{-}}\in \frac{1}{n}\mathbb{Z}_{-}\times \lbrace 0\rbrace)$ be the law of a sample path of $B^{\widetilde{\mathtt{H}}_{n}}$, started at $x+i\varepsilon$, restricted to the event 
$B^{\widetilde{\mathtt{H}}_{n}}_{\zeta_{n}^{-}}\in \frac{1}{n}\mathbb{Z}_{-}\times \lbrace 0\rbrace$ (we do not condition and the total mass is $<1$). Then
\begin{displaymath}
\nu_{\rm exc}^{\widetilde{\mathtt{H}}_{n}}(x\rightarrow (-\infty,0])=
\lim_{\varepsilon\rightarrow 0}\dfrac{1}{\varepsilon}
\mathbb{P}_{x+i\varepsilon}^{\widetilde{\mathtt{H}}_{n}}\left(\cdot, 
B^{\widetilde{\mathtt{H}}_{n}}_{\zeta_{n}^{-}}\in \frac{1}{n}\mathbb{Z}_{-}\times \lbrace 0\rbrace\right).
\end{displaymath}
Let $q\in(1,+\infty)$ and  $x\in ((\frac{1}{n}\mathbb{Z})\cap[1,q])\times \lbrace 0\rbrace$. We will similarly denote by
${\nu_{\rm exc}^{\widetilde{\mathtt{H}}_{n}}(x\rightarrow [1,q])}$ the measure on excursions from $x$ to
$((\frac{1}{n}\mathbb{Z})\cap[1,q])\times \lbrace 0\rbrace$. Let
\begin{equation}
\label{DefExcInfty0}
\nu_{\rm exc}^{\widetilde{\mathtt{H}}_{n}}((-\infty,0]):=
\dfrac{8\pi}{n}\sum_{x\in \frac{1}{n}\mathbb{Z}_{-}\times \lbrace 0\rbrace}
\nu_{\rm exc}^{\widetilde{\mathtt{H}}_{n}}(x\rightarrow (-\infty,0]),
\end{equation}
\begin{equation}
\label{DefExc1q}
\nu_{\rm exc}^{\widetilde{\mathtt{H}}_{n}}([1,q]):=
\dfrac{8\pi}{n}\sum_{x\in ((\frac{1}{n}\mathbb{Z})\cap[1,q])\times \lbrace 0\rbrace}
\nu_{\rm exc}^{\widetilde{\mathtt{H}}_{n}}(x\rightarrow [1,q]).
\end{equation}
$\nu_{\rm exc}^{\widetilde{\mathtt{H}}_{n}}((-\infty,0])$ is a measure on excursions from and to
$\frac{1}{n}\mathbb{Z}_{-}\times \lbrace 0\rbrace$. $\nu_{\rm exc}^{\widetilde{\mathtt{H}}_{n}}([1,q])$ is a measure on excursions from and to $((\frac{1}{n}\mathbb{Z})\cap[1,q])\times \lbrace 0\rbrace$. 

The above measures can be disintegrated over the starting and the endpoint. The measure induced over the couple starting and endpoint is
\begin{displaymath}
8\pi
\sum_{\frac{i}{n}\in~\text{interval}}
\sum_{\frac{j}{n}\in~\text{interval}}
\mathbb{P}_{(\frac{i}{n},\frac{1}{n})}\bigg(B^{\widetilde{\mathtt{H}}_{n}}~\text{hits}
\left(\frac{1}{n}\mathbb{Z}\right)\times \lbrace 0\rbrace~\text{on}~
\left(\frac{j}{n},0\right)\bigg)
\delta_{((\frac{i}{n},\frac{0}{n}),(\frac{j}{n},\frac{0}{n}))},
\end{displaymath}
where "interval" stands for either $(-\infty, 0]$ or $[1,q]$, and 
$\delta_{\cdot}$ denotes the Dirac mass.
Let $G^{\mathtt{H}}(\cdot,\cdot)$ be the Green's function of the simple random walk $(x_{k})_{k\geq 0}$ on $\mathtt{H}=\mathbb{Z}\times\mathbb{N}$, killed at the first hitting time of $\mathbb{Z}\times\lbrace 0\rbrace$. 
Let $i,j\in\mathbb{Z}$. Then
\begin{equation}
\label{EqDL1}
\begin{split}
\mathbb{P}_{(\frac{i}{n},\frac{1}{n})}\bigg(B^{\widetilde{\mathtt{H}}_{n}}~\text{hits}&~
\left(\frac{1}{n}\mathbb{Z}\right)\times \lbrace 0\rbrace~\text{on}~\left(\frac{j}{n},0\bigg)
\right)\\&=
\sum_{k=0}^{+\infty}
\mathbb{P}_{(i,1)}\left(x_{1},\dots,x_{k-1}\not\in
\mathbb{Z}\times\lbrace 0\rbrace
,x_{k}=(j,1),x_{k+1}=(j,0)\right)
\\&=
\dfrac{1}{4}
\sum_{k=0}^{+\infty}
\mathbb{P}_{(i,1)}\left(x_{1},\dots,x_{k-1}\not\in
\mathbb{Z}\times\lbrace 0\rbrace
,x_{k}=(j,1)\right)
\\&=
\dfrac{1}{4}G^{\mathtt{H}}((i,1),(j,1))=
\dfrac{1}{4}G^{\mathtt{H}}((0,1),(j-i,1)).
\end{split}
\end{equation}
Indeed, to go from $(\frac{i}{n},\frac{1}{n})$ to $(\frac{j}{n},0)$ the moving particle needs to reach
$(\frac{j}{n},\frac{1}{n})$, possibly make excursions from and to this point without hitting
$\left(\frac{1}{n}\mathbb{Z}\right)\times \lbrace 0\rbrace$, and then with probability $\frac{1}{4}$ transition to
$(\frac{j}{n},0)$.
Thus, the measure over the starting and endpoint is
\begin{displaymath}
2\pi
\sum_{\frac{i}{n}\in~\text{interval}}
\sum_{\frac{j}{n}\in~\text{interval}}
G^{\mathtt{H}}((0,1),(j-i,1))
\delta_{((\frac{i}{n},\frac{0}{n}),(\frac{j}{n},\frac{0}{n}))}.
\end{displaymath}

Observe that the above measure is invariant by permuting the starting and the endpoint. Moreover, the conditional probability measures on excursions where the both ends are fixed are covariant with time reversal, that is to say the distribution on the unoriented excursion does not change. This means that the whole measures on excursions
$\nu_{\rm exc}^{\widetilde{\mathtt{H}}_{n}}((-\infty,0])$ and
$\nu_{\rm exc}^{\widetilde{\mathtt{H}}_{n}}([1,q])$ are invariant under time reversal. 

According to the asymptotic expansion given in \cite{LawlerLimic2010RW}, Section $8.1.1$,
\begin{equation}
\label{EqDL2}
G^{\mathtt{H}}((0,1),(j,1))=\dfrac{1}{\pi j^{2}}+O\left(\dfrac{1}{j^{3}}\right).
\end{equation}
So, as $n$ tends to infinity, the measure on the starting and endpoint converges to a measure with density with respect to Lebesgue:
\begin{displaymath}
2\dfrac{dx dy}{(y-x)^{2}}
\mathbbm{1}_{x,y\in~\text{interval}}
.
\end{displaymath}
The conditional probability measures on excursions of 
$B^{\widetilde{\mathtt{H}}_{n}}$ with fixed endpoints converge too. The limits are
the probability measures on two-dimensional Brownian excursions from $x$ to $y$ in $\mathbb{H}$, where $x,y\in\mathbb{R}$, and we will denote them
$\mathbb{P}^{\mathbb{H}}_{x,y}(\cdot)$. See \cite{Werner2005ConfRestr}, Section 1.2, for more on these normalised excursion probability measures.

Consequently, as $n$ tends to infinity, 
$\nu_{\rm exc}^{\widetilde{\mathtt{H}}_{n}}((-\infty,0])$ and
$\nu_{\rm exc}^{\widetilde{\mathtt{H}}_{n}}([1,q])$ have limits which are measures on Brownian excursions in $\mathbb{H}$, from and to 
$(-\infty,0]\times \lbrace 0\rbrace$ respectively 
$[1,q]\times \lbrace 0\rbrace$, and which disintegrate as follows:
\begin{displaymath}
\nu_{\rm exc}^{\mathbb{H}}((-\infty,0])=2\int_{-\infty}^{0}\int_{-\infty}^{0}
\mathbb{P}^{\mathbb{H}}_{x,y} \dfrac{dx dy}{(y-x)^{2}},
\qquad
\nu_{\rm exc}^{\mathbb{H}}([1,q])=
2\int_{1}^{q}\int_{1}^{q} \mathbb{P}^{\mathbb{H}}_{x,y} \dfrac{dx dy}{(y-x)^{2}}.
\end{displaymath}
In general, given $a<b\in \mathbb{R}$, we will use the notation
\begin{displaymath}
\nu_{\rm exc}^{\mathbb{H}}([a,b]):=
2\int_{a}^{b}\int_{a}^{b} \mathbb{P}^{\mathbb{H}}_{x,y} \dfrac{dx dy}{(y-x)^{2}}.
\end{displaymath}
See \cite{Werner2005ConfRestr}, Section 4.3, for more on these infinite mass excursion measures.

We will consider on $\widetilde{\mathtt{H}}_{n}$ three independent Poisson point processes:
\begin{itemize}
\item a loop-soup $\mathcal{L}^{\widetilde{\mathtt{H}}_{n}}_{1/2}$,
\item a Poisson point process of excursions of intensity $u\nu_{\rm exc}^{\widetilde{\mathtt{H}}_{n}}((-\infty,0])$, $u>0$,
denoted by $\mathcal{E}^{\widetilde{\mathtt{H}}_{n}}_{u}((-\infty,0])$,
\item a Poisson point process of excursions of intensity $v\nu_{\rm exc}^{\widetilde{\mathtt{H}}_{n}}([1,q])$, $v>0$, denoted by $\mathcal{E}^{\widetilde{\mathtt{H}}_{n}}_{v}([1,q])$.
\end{itemize}
We will consider the following event: either an excursion from $\mathcal{E}^{\widetilde{\mathtt{H}}_{n}}_{u}((-\infty,0])$
intersects an excursion from $\mathcal{E}^{\widetilde{\mathtt{H}}_{n}}_{v}([1,q])$ or an excursion from
$\mathcal{E}^{\widetilde{\mathtt{H}}_{n}}_{u}((-\infty,0])$ and one from
$\mathcal{E}^{\widetilde{\mathtt{H}}_{n}}_{v}([1,q])$ intersect a common cluster of $\mathcal{L}^{\widetilde{\mathtt{H}}_{n}}_{1/2}$. We will denote by
$p^{\widetilde{\mathtt{H}}_{n}}_{1/2,u,v}(q)$ the probability of this event. The second condition of intersecting a common cluster is equivalent to intersecting a common contour in 
$\mathcal{F}_{\rm ext}(\mathcal{L}^{{\widetilde{\mathtt{H}}_{n}}}_{1/2})$. 

Similarly we will consider on $\mathbb{H}$ three independent Poisson point processes:
\begin{itemize}
\item a loop-soup $\mathcal{L}^{\mathbb{H}}_{\alpha}$, $\alpha\in (0,1/2]$,
\item a Poisson point process of excursions of intensity $u\nu_{\rm exc}^{\mathbb{H}}((-\infty,0])$, $u>0$,
denoted by $\mathcal{E}^{\mathbb{H}}_{u}((-\infty,0])$,
\item a Poisson point process of excursions of intensity $v\nu_{\rm exc}^{\mathbb{H}}([1,q])$, $v>0$, denoted by $\mathcal{E}^{\mathbb{H}}_{v}([1,q])$.
\end{itemize}
Then we will consider the event when either an excursion from $\mathcal{E}^{\mathbb{H}}_{u}((-\infty,0])$ intersects an excursion from $\mathcal{E}^{\mathbb{H}}_{v}([1,q])$ or an excursion from $\mathcal{E}^{\mathbb{H}}_{u}((-\infty,0])$ and one from $\mathcal{E}^{\mathbb{H}}_{v}([1,q])$ intersect a common cluster of
$\mathcal{L}^{\mathbb{H}}_{\alpha}$. This event is schematically represented in Figure \ref{CEChain}. We denote
by $p^{\mathbb{H}}_{\alpha,u,v}(q)$ its probability.

\begin{figure}[ht]
\begin{center}
  \includegraphics[scale=0.5]{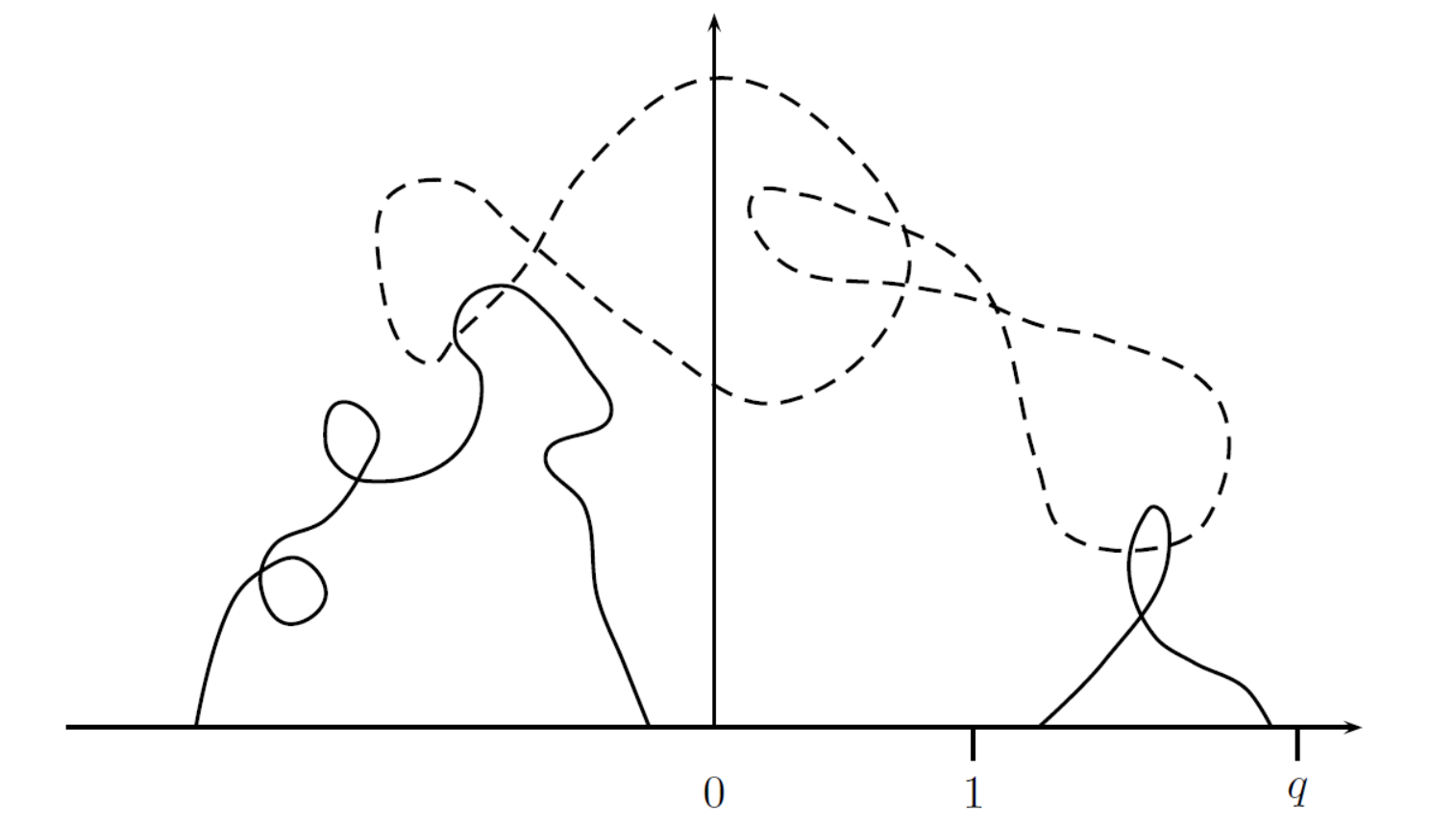}
 \end{center}
  \caption{Two excursions (full lines) connected by a chain of two loops (doted lines).}
  \label{CEChain}
\end{figure}

In Section \ref{secMetricGraph} we will compute $p^{\widetilde{\mathtt{H}}_{n}}_{1/2,u,v}(q)$
using the duality with the Gaussian free field, and compute its limit as $n$ tends to $+\infty$. In Section 
\ref{secContHP}, for an arbitrary value of $v$ and a particular value $u_{0}(\alpha)$ of $u$ (depending on $\alpha$) we will establish a differential equation in $q$ for $1-p^{\mathbb{H}}_{\alpha,u,v}(q)$. Using this we will show that
\begin{equation}
\label{EqConvEq}
\lim_{n\rightarrow +\infty}p^{\widetilde{\mathtt{H}}_{n}}_{1/2,u_{0}(1/2),v}(q)=
p^{\mathbb{H}}_{1/2,u_{0}(1/2),v}(q).
\end{equation}
This convergence will provide the "upper bound" we need. 
Indeed, if the scaling limit of
$\mathcal{F}_{\rm ext}(\mathcal{L}_{1/2}^{\widetilde{\mathtt{H}}_{n}})$ contains contours "strictly larger" than $\hbox{CLE}_{4}$, then
the limit contours would connect
$\mathcal{E}^{\mathbb{H}}_{u_{0}(1/2)}((-\infty,0])$ and
 $\mathcal{E}^{\mathbb{H}}_{v}([1,q])$ with a probability strictly larger than
\\$p^{\mathbb{H}}_{1/2,u_{0}(1/2),v}(q)$, which in
\eqref{EqConvEq} would give a strict inequality rather then an equality.
In Section \ref{secConv} we will prove the convergences to $\hbox{CLE}$
out of \eqref{EqConvEq} using the above argument.

\section{Computations on metric graph}
\label{secMetricGraph}

Let $\mathcal{G}=(V,E)$ be a connected undirected graph. $V$ is countable and each vertex is of finite degree. Each
edge $\lbrace x,y\rbrace$ is endowed with a positive conductance $C(x,y)>0$. We also consider a metric graph 
$\widetilde{\mathcal{G}}$ associated to $\mathcal{G}$ where each edge $\lbrace x,y\rbrace$ is replaced by a continuous line of length
\begin{equation}
\label{LengthConductance}
r(x,y)=\dfrac{1}{2}C(x,y)^{-1}.
\end{equation}

Let $B^{\widetilde{\mathcal{G}}}$ be the Brownian motion on the metric graph $\widetilde{\mathcal{G}}$. Let $F$ be a subset of $V$. Let $\zeta_{F}$ be the first time $B^{\widetilde{\mathcal{G}}}$ hits $F$. Let 
$\mu^{\widetilde{\mathcal{G}},F}$ be the measure on loops associated to 
$(B^{\widetilde{\mathcal{G}}}_{t})_{0\leq t<\zeta_{F}}$, the Brownian motion killed at reaching $F$. It is defined according to \eqref{DefLoopMeasure}. See \cite{Lupu2014LoopsGFF} for details. Let 
$\mathcal{L}^{\widetilde{\mathcal{G}},F}_{\alpha}$ be the Poisson point process of intensity $\alpha\mu^{\widetilde{\mathcal{G}},F}$.

$B^{\widetilde{\mathcal{G}}}$ has a time-space continuous family of local times $L_{t}^{z}(B^{\widetilde{\mathcal{G}}})$.
The Green's function of the killed Brownian motion 
$(B^{\widetilde{\mathcal{G}}}_{t})_{0\leq t<\zeta_{F}}$ is defined to be
\begin{displaymath}
G^{\widetilde{\mathcal{G}},F}(z,z')=
\mathbb{E}_{z}
\left[L_{\zeta_{F}}^{z'}(B^{\widetilde{\mathcal{G}}})\right]
\end{displaymath}
and is symmetric.
Just as $B^{\widetilde{\mathcal{G}}}$, a loop $\gamma\in\mathcal{L}^{\widetilde{\mathcal{G}},F}_{\alpha}$ has a
family of continuous local times $L_{t}^{z}(\gamma)$. We will denote by $t_{\gamma}$ the total life-time of the loop $\gamma$. The occupation field 
\\$(\widehat{\mathcal{L}}^{z}_{\alpha})_{z\in\widetilde{\mathcal{G}}\setminus F}$ is defined as
\begin{displaymath}
\widehat{\mathcal{L}}^{z}_{\alpha}=
\sum_{\gamma\in\mathcal{L}^{\widetilde{\mathcal{G}},F}_{\alpha}}
L_{t_{\gamma}}^{z}(\gamma).
\end{displaymath}
It is a continuous field. The clusters of $\mathcal{L}^{\widetilde{\mathcal{G}},F}_{\alpha}$ are delimited by the zero set of the occupation field.

At intensity parameter $\alpha=1/2$, the occupation field $(\widehat{\mathcal{L}}^{z}_{\alpha})_{z\in\widetilde{\mathcal{G}}\setminus F}$ is related to the Gaussian free field 
$(\phi_{z})_{z\in\widetilde{\mathcal{G}}\setminus F}$ with zero mean and covariance function 
$G^{\widetilde{\mathcal{G}},F}$.  Given $z\in\widetilde{\mathcal{G}}\setminus F$ such that
$\widehat{\mathcal{L}}^{z}_{1/2}>0$, we denote by $\mathcal{C}_{1/2}(z)$ the cluster of
$\mathcal{L}^{\widetilde{\mathcal{G}},F}_{1/2}$ that contains $z$. We introduce a countable family 
$(\sigma(\mathcal{C}_{1/2}(z)))_{z\in\widetilde{\mathcal{G}}\setminus F}$ of
i.i.d. random variables, independent of $\mathcal{L}^{\widetilde{\mathcal{G}},F}_{1/2}$ conditional on the clusters, which equal $-1$ or $1$ with equal probability. There is an equality in law (see \cite{Lupu2014LoopsGFF}):
\begin{equation}
\label{DualityLoopsGFF}
(\phi_{z})_{z\in\widetilde{\mathcal{G}}\setminus F}\stackrel{(d)}{=}
\left(\sigma(\mathcal{C}_{1/2}(z))\sqrt{2\widehat{\mathcal{L}}^{z}_{1/2}}\right)
_{z\in\widetilde{\mathcal{G}}\setminus F}.
\end{equation}

Let $x,y\in V\setminus F$. Let $C^{\rm eq}(x,y),\chi^{\rm eq}_{(x,y)}(x),\chi^{\rm eq}_{(x,y)}(y)$ be the quantities defined by
\begin{displaymath}
\left(
\begin{array}{cc}
G^{\widetilde{\mathcal{G}},F}(x,x) & 
G^{\widetilde{\mathcal{G}},F}(x,y) \\ 
G^{\widetilde{\mathcal{G}},F}(x,y) & 
G^{\widetilde{\mathcal{G}},F}(y,y)
\end{array}
\right)^{-1}
=\left(
\begin{array}{cc}
\chi^{\rm eq}_{(x,y)}(x)+C^{\rm eq}(x,y) & -C^{\rm eq}(x,y) \\ 
-C^{\rm eq}(x,y) & \chi^{\rm eq}_{(x,y)}(y)+C^{\rm eq}(x,y)
\end{array}
\right).
\end{displaymath}
Then $C^{\rm eq}(x,y)>0$, $\chi^{\rm eq}_{(x,y)}(x),
\chi^{\rm eq}_{(x,y)}(y)\geq 0$, 
$(\chi^{\rm eq}_{(x,y)}(x)$ and
$\chi^{\rm eq}_{(x,y)}(y))\neq (0,0)$. 
$C^{\rm eq}(x,y)$, $\chi^{\rm eq}_{(x,y)}(x)$ and
$\chi^{\rm eq}_{(x,y)}(y)$ are the conductances of a network electrically equivalent to $\mathcal{G}$, where all vertices in $F$ are at the same electrical potential. This equivalent network has three vertices, $x$, $y$ and a vertex corresponding to the set $F$. $C^{\rm eq}(x,y)$ is the conductance between $x$ and $y$, $\chi^{\rm eq}_{(x,y)}(x)$ respectively 
$\chi^{\rm eq}_{(x,y)}(y)$ is the conductance
between $x$ and $F$ respectively $y$ and $F$.

Let $\mathcal{N}_{1/2}(x,y)$ the number of loops in $\mathcal{L}^{\widetilde{\mathcal{G}},F}_{1/2}$
that visit both $x$ and $y$.

\begin{lemma}
\label{LemProbNoConx}
Let $u,v>0$ and $x,y\in V\setminus F$.
\begin{equation}
\label{ProbAbsConx}
\mathbb{P}\left(\mathcal{C}_{1/2}(x)\neq\mathcal{C}_{1/2}(y)\Big\vert
\widehat{\mathcal{L}}^{x}_{1/2}=u,
\widehat{\mathcal{L}}^{y}_{1/2}=v,
\mathcal{N}_{1/2}(x,y)=0\right)=e^{-2C^{\rm eq}(x,y)\sqrt{uv}}.
\end{equation}
\end{lemma}

\begin{proof}
If $\mathcal{N}_{1/2}(x,y)>0$ then $\mathcal{C}_{1/2}(x)=\mathcal{C}_{1/2}(y)$. Thus
\begin{multline}
\label{ProbAbsConx2}
\mathbb{P}\left(\mathcal{C}_{1/2}(x)\neq\mathcal{C}_{1/2}(y)\Big\vert
\widehat{\mathcal{L}}^{x}_{1/2}=u,
\widehat{\mathcal{L}}^{y}_{1/2}=v,
\mathcal{N}_{1/2}(x,y)=0\right)\\=
\dfrac{\mathbb{P}\left(\mathcal{C}_{1/2}(x)\neq\mathcal{C}_{1/2}(y)\Big\vert
\widehat{\mathcal{L}}^{x}_{1/2}=u,
\widehat{\mathcal{L}}^{y}_{1/2}=v\right)}
{\mathbb{P}\left(\mathcal{N}_{1/2}(x,y)=0\Big\vert
\widehat{\mathcal{L}}^{x}_{1/2}=u,
\widehat{\mathcal{L}}^{y}_{1/2}=v\right)}.
\end{multline}

The value of the denominator
\begin{displaymath}
\mathbb{P}\left(\mathcal{N}_{1/2}(x,y)=0\Big\vert
\widehat{\mathcal{L}}^{x}_{1/2}=u,
\widehat{\mathcal{L}}^{y}_{1/2}=v\right)
\end{displaymath}
depends only on $u,v$ and on $G^{\widetilde{\mathcal{G}},F}(x,x),G^{\widetilde{\mathcal{G}},F}(y,y),
G^{\widetilde{\mathcal{G}},F}(x,y)$ (or equivalently on 
\\$C^{\rm eq}(x,y)$, $\chi^{\rm eq}_{(x,y)}(x)$, 
$\chi^{\rm eq}_{(x,y)}(y)$).
This a general property of the loop-soups (see \cite{LeJan2011Loops}, especially chapter $7$).

As for the numerator, it can be computed using the duality with the Gaussian free field \eqref{DualityLoopsGFF}. 
If $\mathcal{C}_{1/2}(x)=\mathcal{C}_{1/2}(y)$,
then $\phi_{x}$ and $\phi_{y}$ have same sign. Otherwise, $\phi_{x}$ and $\phi_{y}$ have same sign with conditional probability $1/2$.
Thus
\begin{equation*}
\begin{split}
\mathbb{P}\Big(\mathcal{C}_{1/2}(x)\neq\mathcal{C}_{1/2}(y)\Big\vert
\widehat{\mathcal{L}}^{x}_{1/2}=u, 
\widehat{\mathcal{L}}^{y}_{1/2}=v\Big)
&=
1-\mathbb{E}\left[\operatorname{sgn}(\phi_{x})\operatorname{sgn}(\phi_{y})\big\vert\vert\phi_{x}\vert
=\sqrt{2u},\vert\phi_{y}\vert=\sqrt{2v}\right]
\\&=
1-
\dfrac{e^{2C^{\rm eq}(x,y)\sqrt{uv}}-e^{-2C^{\rm eq}(x,y)\sqrt{uv}}}
{e^{2C^{\rm eq}(x,y)\sqrt{uv}}+e^{-2C^{\rm eq}(x,y)\sqrt{uv}}}
\\&=\dfrac{e^{-2C^{\rm eq}(x,y)\sqrt{uv}}}{\cosh(2C^{\rm eq}(x,y)\sqrt{uv})}.
\end{split}
\end{equation*}

It follows that the probability \eqref{ProbAbsConx} that we want to compute only depends on $u,v$ and on
\\$C^{\rm eq}(x,y),\chi^{\rm eq}_{(x,y)}(x),\chi^{\rm eq}_{(x,y)}(y)$. Thus it is the same if we replace $\widetilde{\mathcal{G}}$
by the interval
\begin{displaymath}
I=\left(-\frac{1}{2}\chi^{\rm eq}_{(x,y)}(x)^{-1},
\frac{1}{2}C^{\rm eq}(x,y)^{-1}
+\dfrac{1}{2}\chi^{\rm eq}_{(x,y)}(y)^{-1}\right),
\end{displaymath}
the Brownian motion on $\widetilde{\mathcal{G}}$ by the Brownian motion on $I$ killed at endpoints, and the points $x$ and $y$ by $0$ and $\tfrac{1}{2}C^{\rm eq}(x,y)^{-1}$ respectively. 
According to Lemma $3.4$ and $3.5$ in
\cite{Lupu2014LoopsGFF}, we get \eqref{ProbAbsConx}.

By the way we also get that
\begin{displaymath}
\mathbb{P}\left(\mathcal{N}_{1/2}(x,y)=0\Big\vert
\widehat{\mathcal{L}}^{x}_{1/2}=u,
\widehat{\mathcal{L}}^{y}_{1/2}=v\right)=
\cosh(2C^{\rm eq}(x,y)\sqrt{uv})^{-1}.
\qedhere
\end{displaymath}
\end{proof}

In \cite{LeJan2011Loops}, chapter $7$, there is a combinatorial representation of $C^{\rm eq}(x,y)$. Given $z\in V$, we will denote
\begin{displaymath}
\lambda(z):=\sum_{\substack{z'\in V\\ z'\sim z}}C(z,z'),
\end{displaymath}
where the sum is over the neighbours of $z$ in the (discrete) graph $\mathcal{G}$. Then
\begin{displaymath}
C^{\rm eq}(x,y)=\lambda(x)\sum_{j\geq 1}\sum_{
\substack{
(z_{0},\dots,z_{j})\in (V\setminus F)^{j+1} \\ 
z_{0}=x,z_{j}=y,z_{i}\sim z_{i-1}\\ 
z_{i}\neq x,y~\text{for}~1\leq i\leq j-1
}
}\prod_{i=1}^{j}\dfrac{C(z_{i-1},z_{i})}{\lambda(z_{i-1})}.
\end{displaymath}
The sum is over all the discrete nearest neighbour paths joining $x$ to $y$, that avoid $F$ and only visit $x$ and $y$ at endpoints. The above equality can be rewritten as
\begin{equation}
\label{RepCeq}
C^{\rm eq}(x,y)=\sum_{\substack{z\in V\\z\sim x}}C(x,z)\mathbb{P}_{z}(B^{\widetilde{\mathcal{G}}}~\text{hits}~
y~\text{before}~F~\text{or}~x).
\end{equation}

Next we return to the metric graph half-plane 
$\widetilde{\mathtt{H}}_{n}$. Let $a>0$. Let 
$\widetilde{\mathcal{G}}_{n,a}(q)$ be the metric graph obtained from $\widetilde{\mathtt{H}}_{n}$ by identifying the following vertices:
\begin{itemize}
\item All the vertices in 
$((\frac{1}{n}\mathbb{Z})\cap[-a,0])\times \lbrace 0\rbrace$ 
are identified into a single vertex
$\lhd_{n}(a)$.
\item All the vertices in 
$((\frac{1}{n}\mathbb{Z})\cap[1,q])\times \lbrace 0\rbrace$ 
are identified into a single vertex
$\rhd_{n}(q)$.
\end{itemize}
See Figure \ref{CEGraph}. We consider a finite value of $a$ just to have a finite degree for the quotient vertex $\lhd_{n}(a)$, but eventually we will consider $a\to +\infty$.

\begin{figure}[ht]
\begin{center}
  \includegraphics[scale=0.5]{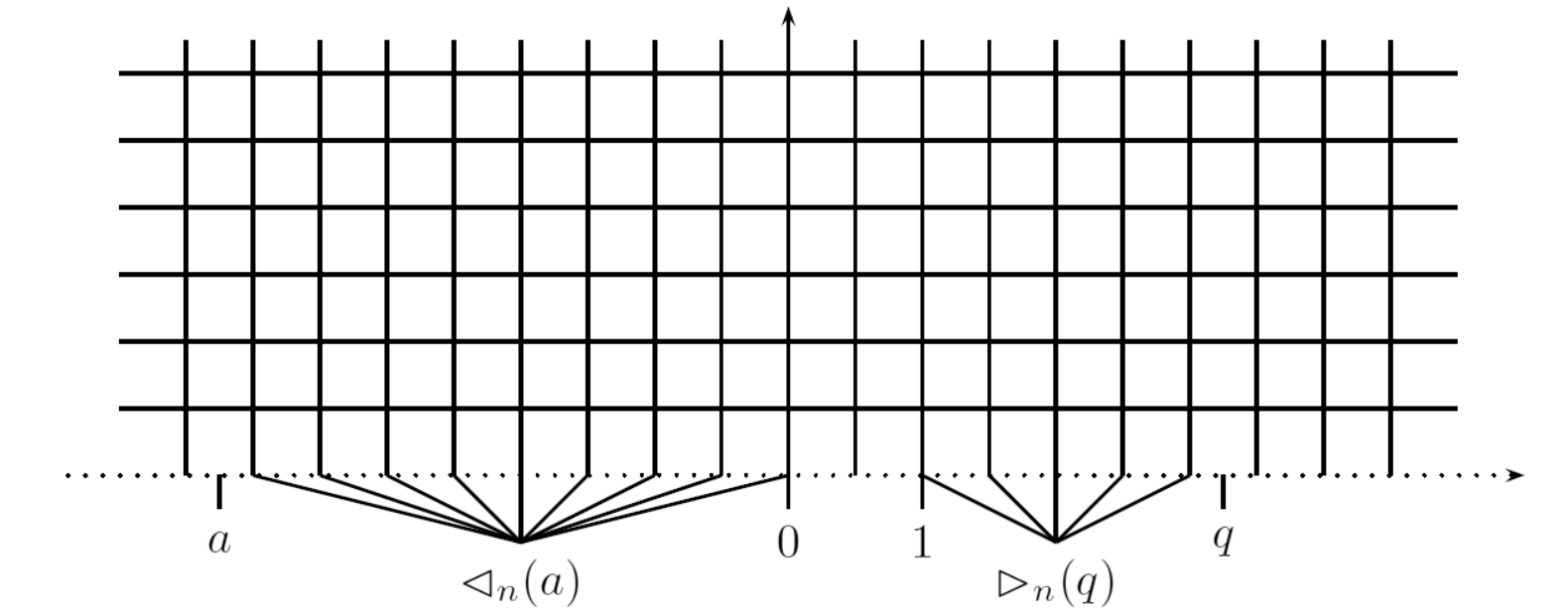}
\end{center}
  \caption{Illustration of points identified into $\lhd_{n}(a)$ and $\rhd_{n}(q)$.}
  \label{CEGraph}
\end{figure}

As the length of the line joining $(\frac{i}{n},\frac{j}{n})$ to $(\frac{i+1}{n},\frac{j}{n})$ or
$(\frac{i}{n},\frac{j}{n})$ to $(\frac{i}{n},\frac{j+1}{n})$ is $\frac{1}{n}$, the corresponding conductance is according
to \eqref{LengthConductance} equal to $\frac{n}{2}$. Let 
$C^{\rm eq}_{n,a}(q)$ be the equivalent (or effective) conductance between
$\lhd_{n}(a)$ and $\rhd_{n}(q)$ when all the points in 
$(\frac{1}{n})\mathbb{Z}\times\lbrace 0\rbrace$ 
other than those identified to $\lhd_{n}(a)$ or $\rhd_{n}(q)$ have the same electrical potential. According
to \eqref{RepCeq},
\begin{displaymath}
C^{\rm eq}_{n,a}(q)=\dfrac{n}{2}\sum_{i=n}^{\lfloor nq\rfloor}
\mathbb{P}_{(\frac{i}{n},\frac{1}{n})}\left(B^{\widetilde{\mathtt{H}}_{n}}~\text{hits}~
\left(\frac{1}{n}\mathbb{Z}\right)\times \lbrace 0\rbrace~\text{on}~[-a,0]
\times\lbrace 0\rbrace\right).
\end{displaymath}
As $a$ tends to infinity, $C^{\rm eq}_{n,a}(q)$ increases and converges to
\begin{equation}
\label{Ceqn}
C^{\rm eq}_{n}(q)=\dfrac{n}{2}\sum_{i=n}^{\lfloor nq\rfloor}
\mathbb{P}_{(\frac{i}{n},\frac{1}{n})}\left(B^{\widetilde{\mathtt{H}}_{n}}~\text{hits}~
\left(\frac{1}{n}\mathbb{Z}\right)\times \lbrace 0\rbrace~\text{on}~(-\infty,0]
\times\lbrace 0\rbrace\right).
\end{equation}

\begin{lemma}
\label{LemCeqHP}
For all $n\in\mathbb{N}^{\ast}$ and $q>1$, $C^{\rm eq}_{n}(q)<+\infty$. Moreover,
\begin{displaymath}
\lim_{n\rightarrow +\infty}\dfrac{1}{n}C^{\rm eq}_{n}(q)
=\dfrac{1}{8\pi}\log(q).
\end{displaymath}
\end{lemma}

\begin{proof}
Using the computation \eqref{EqDL1} and the asymptotic expansion
\eqref{EqDL2}, we get that 
\\$C^{\rm eq}_{n}(q)<+\infty$ and that
\begin{equation*}
\begin{split}
\dfrac{1}{n}C^{\rm eq}_{n}(q)=&\dfrac{1}{8\pi}\sum_{i=n}^{\lfloor nq\rfloor}
\sum_{j=0}^{+\infty}\dfrac{1}{(i+j)^{2}}+
O\left(\sum_{i=n}^{\lfloor nq\rfloor}
\sum_{j=0}^{+\infty}\dfrac{1}{(i+j)^{3}}\right)\\
=&\dfrac{1}{8\pi}\sum_{i=n}^{\lfloor nq\rfloor}
\dfrac{1}{i}+O\left(\sum_{i=n}^{\lfloor nq\rfloor}
\dfrac{1}{i^{2}}\right)\\
=&\dfrac{1}{8\pi}\log(q)+O\left(\dfrac{1}{n}\right).
\qedhere
\end{split}
\end{equation*}
\end{proof}

Let $\nu_{\rm exc}^{\widetilde{\mathtt{H}}_{n}}([-a,0])$ be the measure on excursions $\nu_{\rm exc}^{\widetilde{\mathtt{H}}_{n}}((-\infty,0])$ restricted to the excursions from and to
$[-a,0]\times\lbrace 0\rbrace$. Let $\mathcal{L}_{\alpha}^{\widetilde{\mathcal{G}}_{n,a}(q)}$ be the loop-soup associated to the Brownian motion on the metric graph $\widetilde{\mathcal{G}}_{n,a}(q)$, killed at the first hitting time of 
$(\frac{1}{n})\mathbb{Z}\times\lbrace 0\rbrace$ outside the points identified to $\lhd_{n}(a)$ or $\rhd_{n}(q)$. Let $(\widehat{\mathcal{L}}^{z}_{n,a,q,\alpha})_{z\in\widetilde{\mathcal{G}}_{n,a}(q)}$  be the occupation field of $\mathcal{L}_{\alpha}^{\widetilde{\mathcal{G}}_{n,a}(q)}$. Let $\mathcal{N}_{\alpha}(\lhd_{n}(a),\rhd_{n}(q))$ be the number of loops in $\mathcal{L}_{\alpha}^{\widetilde{\mathcal{G}}_{n,a}(q)}$ joining $\lhd_{n}(a)$ to $\rhd_{n}(q)$.

\begin{lemma}
\label{LemRepLoopsExc}
Let $a,\alpha,u,v>0$. We consider $\mathcal{L}_{\alpha}^{\widetilde{\mathcal{G}}_{n,a}(q)}$ conditioned on
\begin{displaymath}
\widehat{\mathcal{L}}^{\lhd_{n}(a)}_{n,a,q,\alpha}=u,
\widehat{\mathcal{L}}^{\rhd_{n}(q)}_{n,a,q,\alpha}=v~\text{and}~
\mathcal{N}_{\alpha}(\lhd_{n}(a),\rhd_{n}(q))=0.
\end{displaymath}
Then $\mathcal{L}_{\alpha}^{\widetilde{\mathcal{G}}_{n,a}(q)}$ consists of three independent families of loops:
\begin{itemize}
\item The loops that visit neither $\lhd_{n}(a)$ nor $\rhd_{n}(q)$. These are the same as the loops in
$\mathcal{L}^{\widetilde{\mathtt{H}}_{n}}_{\alpha}$.
\item The loops that visit $\lhd_{n}(a)$. The excursions these loops make outside $\lhd_{n}(a)$ form a Poisson point process of intensity 
$\frac{n}{8\pi}u\nu_{\rm exc}^{\widetilde{\mathtt{H}}_{n}}([-a,0])$.
\item The loops that visit $\rhd_{n}(q)$. The excursions these loops make outside $\rhd_{n}(q)$ form a Poisson point process of intensity $\frac{n}{8\pi}v\nu_{\rm exc}^{\widetilde{\mathtt{H}}_{n}}([1,q])$.
\end{itemize}
\end{lemma}

\begin{proof}
This follows from universal properties of loop-soups.
The subset of loops that do not visit a given set $F'$ is distributed like the loop-soup of the same Markov process, but 
with additional killing at hitting $F'$ (restriction property). The loops that visit a particular point $z$ can be represented by a Poisson point process of Markovian excursions outside $z$.
See for instance \cite{LeJan2011Loops}, Sections 2.2, 2.3, 7.1, 7.2, 7.3, and \cite{LawlerLimic2010RW}, 
Propositions 9.3.1 and 9.4.1. 
The factor $\frac{n}{8\pi}$ in $\frac{n}{8\pi}u\nu_{\rm exc}^{\widetilde{\mathtt{H}}_{n}}([-a,0])$ and 
$\frac{n}{8\pi}v\nu_{\rm exc}^{\widetilde{\mathtt{H}}_{n}}([1,q])$ comes from the normalisation factor $\frac{8\pi}{n}$ in the definition of $\nu_{\rm exc}^{\widetilde{\mathtt{H}}_{n}}([-a,0])$ (\eqref{DefExcInfty0}) and $\nu_{\rm exc}^{\widetilde{\mathtt{H}}_{n}}([1,q])$ (\eqref{DefExc1q}).
\end{proof}

\begin{proposition}
\label{PropProbMetrGraph}
Let $u,v>0$, $q>1$ and $n\geq 1$.
\begin{equation}
\label{ProbCondCable}
p^{\widetilde{\mathtt{H}}_{n}}_{1/2,u,v}(q)=1-e^{-2C^{\rm eq}_{n}(q)\frac{8\pi\sqrt{uv}}{n}}.
\end{equation}
\begin{equation}
\label{LimProbCable}
\lim_{n\rightarrow +\infty}p^{\widetilde{\mathtt{H}}_{n}}_{1/2,u,v}(q)=
1-q^{-2\sqrt{uv}}.
\end{equation}
\end{proposition}

\begin{proof}
Let $a>0$. Consider three independent Poisson point processes:
\begin{itemize}
\item a loop-soup $\mathcal{L}^{\widetilde{\mathtt{H}}_{n}}_{1/2}$,
\item a P.p.p of excursions of intensity $u\nu_{\rm exc}^{\widetilde{\mathtt{H}}_{n}}([-a,0])$,
\item a P.p.p of excursions of intensity $v\nu_{\rm exc}^{\widetilde{\mathtt{H}}_{n}}([1,q])$.
\end{itemize}
The probability for the two P.p.p. of excursions to be connected either directly or through a cluster of
$\mathcal{L}^{\widetilde{\mathtt{H}}_{n}}_{1/2}$ equals, according to Lemma \ref{LemRepLoopsExc}, the
probability for $\lhd_{n}(a)$ and $\rhd_{n}(q)$ to be in the same cluster of
$\mathcal{L}_{1/2}^{\widetilde{\mathcal{G}}_{n,a}(q)}$ conditional on
$\widehat{\mathcal{L}}^{\lhd_{n}(a)}_{n,a,q,1/2}=\frac{8\pi}{n}u$,
$\widehat{\mathcal{L}}^{\rhd_{n}(q)}_{n,a,q,1/2}=\frac{8\pi}{n}v$ and
$\mathcal{N}_{1/2}(\lhd_{n}(a),\rhd_{n}(q))=0$. According to Lemma \ref{LemProbNoConx} this probability equals
\begin{displaymath}
1-e^{-2C^{\rm eq}_{n,a}(q)\frac{8\pi\sqrt{uv}}{n}}.
\end{displaymath}
Taking the limit as $a$ tends to infinity we get \eqref{ProbCondCable}. Using Lemma \ref{LemCeqHP} we get the limit
\eqref{LimProbCable}.
\end{proof}

\section{Computations on continuum half-plane}
\label{secContHP}

On the continuum upper half plane $\mathbb{H}$ we consider two independent Poisson point processes:
\begin{itemize}
\item a Brownian loop-soup $\mathcal{L}^{\mathbb{H}}_{\alpha}$, $0<\alpha\leq 1/2$,
\item a P.p.p. of Brownian excursions from and to 
$(-\infty,0]\times\lbrace 0\rbrace$, 
$\mathcal{E}^{\mathbb{H}}_{u}((-\infty,0])$, $u>0$. 
\end{itemize}
We will consider the clusters made out of loops in $\mathcal{L}^{\mathbb{H}}_{\alpha}$ and excursions in $\mathcal{E}^{\mathbb{H}}_{u}((-\infty,0])$. Among these clusters we only take the clusters that contain at least one excursion and consider the rightmost envelop of these clusters. This envelop is a non self-intersecting curve joining $\mathbb{R}$ to infinity. It can be formally defined as follows. Take the clusters that contain at least one excursion.
The curve minus its starting point on $\mathbb{R}$ is the right-most component of the boundary in $\mathbb{H}$ of the closure in $\mathbb{H}$ of the set of points visited by the above clusters.

All the excursions $\mathcal{E}^{\mathbb{H}}_{u}((-\infty,0])$ are located left to the curve and there are only clusters made of loops right to it. According to \cite{Werner2003SLeasBound} and \cite{WernerWu2013FromCLEtoSLE} this boundary curve is an $\hbox{SLE}(\kappa,\rho)$ starting from $0$, where $\kappa$ is given by \eqref{ckappa} and $\rho$ by
\begin{displaymath}
u=\dfrac{(\rho+2)(\rho+6-\kappa)}{4\kappa}.
\end{displaymath}
We will define
\begin{equation}
\label{Defu0c}
u_{0}(\alpha):=\dfrac{6-\kappa(\alpha)}{2\kappa(\alpha)}.
\end{equation}
We will consider the particular case
$u=u_{0}(\alpha)$ (and  thus $\rho=0$), 
which is simpler to deal with.
$\hbox{SLE}(\kappa,\rho)$ is then a chordal $\hbox{SLE}_{\kappa}$ curve starting from $0$. For a description of $\hbox{SLE}$ processes see \cite{Werner2004SLEStFlour}. We will denote by $(\xi_{t})_{t\geq 0}$ this curve. $\xi_{0}=0$. It does not touch $\mathbb{R}$ at positive times. See Figure \ref{CEInterface}.

\begin{figure}[ht]
\begin{center}
  \includegraphics[scale=0.5]{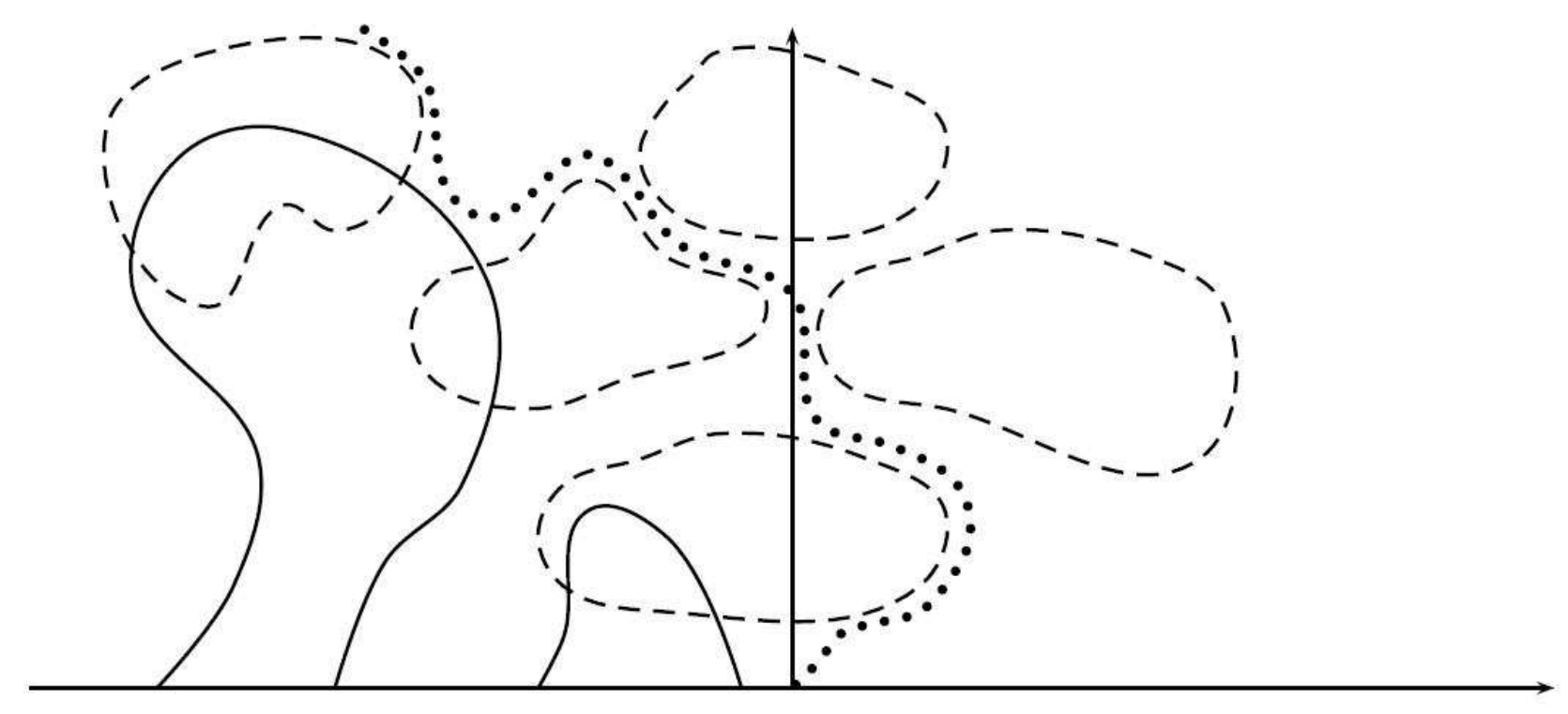}
  \end{center}
  \caption{Full lines represent Brownian excursions in $\mathcal{E}^{\mathbb{H}}_{u}((-\infty,0])$.
  Dashed lines represent
contours in 
$\mathcal{F}_{\rm ext}(\mathcal{L}_{\alpha}^{\mathbb{H}})$.
The dotted line represents $\xi$.}
  \label{CEInterface}
\end{figure}

There is only one conformal map $g_{t}$ that sends $\mathbb{H}\setminus \xi([0,t])$ (half-plane minus the curve up to time $t$) onto $\mathbb{H}$ and that is normalised at infinity $z\rightarrow \infty$ as
\begin{displaymath}
g_{t}(z)=z+\dfrac{a_{t}}{z}+o(z^{-1}).
\end{displaymath}
Moreover, one parametrises the curve by half-plane capacity 
($a_{t}= 2t$).
The Loewner flow $(g_{t})_{t\geq 0}$ satisfies the differential equation
\begin{displaymath}
\dfrac{\partial g_{t}(z)}{\partial t}=\dfrac{2}{g_{t}(z)-\sqrt{\kappa} W_{t}},
\end{displaymath}
where $(W_{t})_{t\geq 0}$ is a standard Brownian motion on $\mathbb{R}$.

\begin{lemma}
\label{LemSLE}
Let $\alpha\in (0,1/2]$.
$p^{\mathbb{H}}_{\alpha,u_{0}(\alpha),v}(q)$ equals the probability that an excursion from 
$\mathcal{E}^{\mathbb{H}}_{v}([1,q])$ intersects an independent 
$\hbox{SLE}_{\kappa(\alpha)}$ curve.
\end{lemma}

\begin{proof}
Let $\xi$ be the $\hbox{SLE}_{\kappa(\alpha)}$ 
curve constructed from
$\mathcal{L}_{\alpha}^{\mathbb{H}}$ and
$\mathcal{E}^{\mathbb{H}}_{u_{0}(\alpha)}((-\infty,0])$, independent from 
$\mathcal{E}^{\mathbb{H}}_{v}([1,q])$.
If no excursion from $\mathcal{E}^{\mathbb{H}}_{v}([1,q])$
intersects $\xi$, then these excursions are all on the right side of $\xi$ and by definition of $\xi$, can only intersects loops in $\mathcal{L}_{\alpha}^{\mathbb{H}}$ that are not connected to 
$\mathcal{E}^{\mathbb{H}}_{u_{0}(\alpha)}((-\infty,0])$.

Conversely, assume that an excursion $\gamma$ from 
$\mathcal{E}^{\mathbb{H}}_{v}([1,q])$
intersects $\xi$ at a point $z_{0}$. Then $\Im(z_{0})>0$.
Since $\mathcal{E}^{\mathbb{H}}_{v}([1,q])$ and 
$\xi$ are independent, by the properties of sample Brownian paths, there is $\varepsilon>0$ small enough such that $\gamma$ makes a closed loop around the disc with center $z_{0}$ and radius $\varepsilon$, disconnecting it from infinity. Thus, any connected set that intersects both this disc and the real line, has to intersect $\gamma$. By the definition of $\xi$, there is either an excursion from 
$\mathcal{E}^{\mathbb{H}}_{u_{0}(\alpha)}((-\infty,0])$, or
a loop from $\mathcal{L}_{\alpha}^{\mathbb{H}}$ connected by a finite chain to an excursion from 
$\mathcal{E}^{\mathbb{H}}_{u_{0}(\alpha)}((-\infty,0])$, 
that intersects the $\varepsilon$-neighbourhood of $z_{0}$. Denote this excursion or loop by $\gamma'$. In the first case, the excursion $\gamma'$ intersects 
$\gamma$. In the second case, an element from the chain connecting $\gamma'$
to $\mathcal{E}^{\mathbb{H}}_{u_{0}(\alpha)}((-\infty,0])$
intersects $\gamma$.
\end{proof}

The excursions $\mathcal{E}^{\mathbb{H}}_{v}([1,q])$ satisfy the one-sided conformal restriction property 
(see \cite{Werner2005ConfRestr}, Section 8, and 
\cite{Werner2005ConfRestr}, Section 4, in particular Section 4.3 ): if $K$ is a compact subset of $\mathbb{C}$ that does not intersect $[1,q]\times\lbrace 0\rbrace$ and such that $\mathbb{H}\setminus K$ is simply connected, if $f$ is a conformal map from $\mathbb{H}\setminus K$ onto $\mathbb{H}$ such that $f(1)<f(q)\in\mathbb{R}$, then the probability that $\mathcal{E}^{\mathbb{H}}_{v}([1,q])$ does not intersect $K$ equals
\begin{displaymath}
\left(\dfrac{f'(1)f'(q)(q-1)^{2}}{(f(q)-f(1))^{2}}\right)^{v}.
\end{displaymath}
Moreover, conditional on this event, the law of $f(\mathcal{E}^{\mathbb{H}}_{v}([1,q]))$
is $\mathcal{E}^{\mathbb{H}}_{v}([f(1),f(q)])$, up to a change of parametrisation of the excursions. From this conformal restriction property, it immediately follows:

\begin{lemma}
\label{LemCondSLErestr}
Let $\kappa\in(0,4]$. Let $(\xi_{t})_{t\geq 0}$ be an $\hbox{SLE}_{\kappa}$ with the driving Brownian motion 
$(\sqrt{\kappa}W_{t})_{t\geq 0}$ and Loewner flow $(g_{t})_{t\geq 0}$.  Denote by $g'_{t}$ the derivative of $g_{t}$ with respect the complex variable:
\begin{displaymath}
g'_{t}(z)=\dfrac{\partial g_{t}(z)}{\partial z}.
\end{displaymath}
Denote by $\bar{p}_{\kappa,v}(q)$ the probability that an independent family of excursions $\mathcal{E}^{\mathbb{H}}_{v}([1,q])$ does not intersect $\xi$. Then the conditional probability of the event that $\mathcal{E}^{\mathbb{H}}_{v}([1,q])$ does not intersect $\xi$ conditional on $(\xi_{s})_{0\leq s\leq t}$ (or equivalently conditional on
$(W_{s})_{0\leq s\leq t}$) and on not intersecting $(\xi_{s})_{0\leq s\leq t}$ equals
\begin{equation}
\label{CondProbSLE1}
\bar{p}_{\kappa,v}\left(\dfrac{g_{t}(q)-\sqrt{\kappa} W_{t}}{g_{t}(1)-\sqrt{\kappa} W_{t}}\right).
\end{equation}
The conditional probability of the event that $\mathcal{E}^{\mathbb{H}}_{v}([1,q])$ does not intersect $\xi$ conditional on $(\xi_{s})_{0\leq s\leq t}$ is
\begin{equation}
\label{CondProbSLE2}
\left(\dfrac{g_{t}'(1)g_{t}'(q)(q-1)^{2}}{(g_{t}(q)-g_{t}(1))^{2}}\right)^{v}
\bar{p}_{\kappa,v}\left(\dfrac{g_{t}(q)-\sqrt{\kappa} W_{t}}{g_{t}(1)-\sqrt{\kappa} W_{t}}\right).
\end{equation}
In particular, for all $t\geq 0$,
\begin{equation}
\label{IntegralEq}
\bar{p}_{\kappa,v}(q)=\mathbb{E}\left[\left(\dfrac{g_{t}'(1)g_{t}'(q)(q-1)^{2}}{(g_{t}(q)-g_{t}(1))^{2}}\right)^{v}
\bar{p}_{\kappa,v}\left(\dfrac{g_{t}(q)-\sqrt{\kappa} W_{t}}{g_{t}(1)-\sqrt{\kappa} W_{t}}\right)\right].
\end{equation}
\end{lemma}

\begin{proof}
\eqref{CondProbSLE1} is the conditional probability that 
$g_{t}(\mathcal{E}^{\mathbb{H}}_{v}([1,q]))$ does not intersect 
\\$(g_{t}(\xi_{t+s}))_{s\geq 0}$. To express it we used the fact that $g_{t}(\mathcal{E}^{\mathbb{H}}_{v}([1,q]))$ has same law as 
\\$\mathcal{E}^{\mathbb{H}}_{v}([g_{t}(1),g_{t}(q)])$
and that $(g_{t}(\xi_{t+s}))_{s\geq 0}$ is a chordal $\hbox{SLE}_{\kappa}$ starting from $\sqrt{\kappa} W_{t}$. In 
\eqref{CondProbSLE2} we multiplied the conditional probability that $\mathcal{E}^{\mathbb{H}}_{v}([1,q])$ does not intersect $(\xi_{s})_{0\leq s\leq t}$ and the conditional probability that $g_{t}(\mathcal{E}^{\mathbb{H}}_{v}([1,q]))$
does not intersect $(g_{t}(\xi_{t+s}))_{s\geq 0}$.
\end{proof}

Next we derive the differential equation in $q$ satisfied by 
$\bar{p}_{\kappa,v}(q)$ on $(1,+\infty)$, provided 
$\bar{p}_{\kappa,v}$ is $\mathcal{C}^{2}$-regular.

\begin{lemma}
\label{LemDiffEq}
Let $\kappa\in(0,4]$, $v>0$ and $q>1$. Let $f$ be a bounded, $\mathcal{C}^{2}$ function on $(1,+\infty)$.
Then
\begin{displaymath}
\left(\dfrac{g_{t}'(1)g_{t}'(q)(q-1)^{2}}{(g_{t}(q)-g_{t}(1))^{2}}\right)^{v}
f\left(\dfrac{g_{t}(q)-\sqrt{\kappa} W_{t}}{g_{t}(1)-\sqrt{\kappa} W_{t}}\right)
\end{displaymath}
is a martingale if and only if $f$ satisfies the differential equation
\begin{equation}
\label{EqDiffSLEConfRestr}
f''+\dfrac{1}{(q-1)q}\left(\left(2-\dfrac{4}{\kappa}\right)q-\dfrac{4}{\kappa}\right)f'-\dfrac{4v}{\kappa q^{2}}f=0.
\end{equation}
\end{lemma}

\begin{proof}
Let
\begin{equation}
\label{NotatRtqt}
R_{t}:=\dfrac{g_{t}'(1)g_{t}'(q)(q-1)^{2}}{(g_{t}(q)-g_{t}(1))^{2}},
\qquad
q_{t}:=\dfrac{g_{t}(q)-\sqrt{\kappa} W_{t}}{g_{t}(1)-\sqrt{\kappa} W_{t}}.
\end{equation}
$R_{t}$ has bounded variation (in $t$). Let
\begin{displaymath}
M_{t}:=R_{t}^{v}f(q_{t}).
\end{displaymath}
We apply Itô's formula to $(M_{t})_{t\geq 0}$.
\begin{displaymath}
dM_{t}=R_{t}^{v}\left(v f(q_{t})\dfrac{dR_{t}}{R_{t}}+
f'(q_{t}) dq_{t}+\dfrac{1}{2}f''(q_{t}) d\langle q\rangle_{t}\right).
\end{displaymath}

Denote $\overline{\mathbb{H}}:=\lbrace\Im(z)\geq 0\rbrace$. For $z\in\overline{\mathbb{H}}\setminus \xi([0,t])$,
\begin{displaymath}
\dfrac{\partial g'_{t}(z)}{\partial t}=\dfrac{\partial}{\partial z}\left(\dfrac{\partial g_{t}(z)}{\partial t}\right)=
\dfrac{\partial}{\partial z}\left(\dfrac{2}{g_{t}(z)-\sqrt{\kappa} W_{t}}\right)=
\dfrac{-2g'_{t}(z)}{(g_{t}(z)-\sqrt{\kappa} W_{t})^{2}}.
\end{displaymath}
Thus
\begin{equation*}
\begin{split}
dR_{t}=&\bigg(\dfrac{-2g'_{t}(1)g'_{t}(q)(q-1)^{2}}{(g_{t}(1)-\sqrt{\kappa} W_{t})^{2}(g_{t}(q)-g_{t}(1))^{2}}
+\dfrac{-2g'_{t}(1)g'_{t}(q)(q-1)^{2}}{(g_{t}(q)-\sqrt{\kappa} W_{t})^{2}(g_{t}(q)-g_{t}(1))^{2}}
\\&+\dfrac{4g_{t}'(1)g_{t}'(q)(q-1)^{2}}{(g_{t}(1)-\sqrt{\kappa} W_{t})(g_{t}(q)-g_{t}(1))^{3}}
+\dfrac{-4g_{t}'(1)g_{t}'(q)(q-1)^{2}}{(g_{t}(q)-\sqrt{\kappa} W_{t})(g_{t}(q)-g_{t}(1))^{3}}\bigg) dt
\\=&-2R_{t}\bigg(\dfrac{1}{(g_{t}(1)-\sqrt{\kappa} W_{t})^{2}}+
\dfrac{1}{(g_{t}(q)-\sqrt{\kappa} W_{t})^{2}}\\&-
\dfrac{2}{(g_{t}(1)-\sqrt{\kappa} W_{t})(g_{t}(q)-\sqrt{\kappa} W_{t})}\bigg) dt
\\=&-2R_{t}\left(\dfrac{1}{g_{t}(1)-\sqrt{\kappa} W_{t}}-
\dfrac{1}{g_{t}(q)-\sqrt{\kappa} W_{t}}\right)^{2} dt
\\=&-2R_{t}\dfrac{(q_{t}-1)^{2}}{(g_{t}(q)-\sqrt{\kappa} W_{t})^{2}} dt.
\end{split}
\end{equation*}

Further
\begin{equation*}
\begin{split}
dq_{t}=&\sqrt{\kappa}\left(\dfrac{-1}{g_{t}(1)-\sqrt{\kappa} W_{t}}+
\dfrac{g_{t}(q)-\sqrt{\kappa} W_{t}}{(g_{t}(1)-\sqrt{\kappa} W_{t})^{2}}\right) dW_{t}
\\&+\bigg(\dfrac{2}{(g_{t}(q)-\sqrt{\kappa} W_{t})(g_{t}(1)-\sqrt{\kappa} W_{t})}-
2\dfrac{g_{t}(q)-\sqrt{\kappa} W_{t}}{(g_{t}(1)-\sqrt{\kappa} W_{t})^{3}}\\&
+\kappa\dfrac{g_{t}(q)-g_{t}(1)}{(g_{t}(1)-\sqrt{\kappa} W_{t})^{3}}\bigg) dt
\\=&\dfrac{\sqrt{\kappa}(q_{t}-1)q_{t}}{g_{t}(q)-\sqrt{\kappa} W_{t}} dW_{t}
+\dfrac{(q_{t}-1)q_{t}}{(g_{t}(q)-\sqrt{\kappa} W_{t})^{2}}((\kappa-2)q_{t}-2) dt.
\end{split}
\end{equation*}

\begin{displaymath}
d\langle q\rangle_{t}=\dfrac{\kappa (q_{t}-1)^{2}q_{t}^{2}}{(g_{t}(q)-\sqrt{\kappa} W_{t})^{2}} dt.
\end{displaymath}

Finally,
\begin{equation*}
\begin{split}
dM_{t}=&R_{t}^{v}f'(q_{t})\dfrac{\sqrt{\kappa}(q_{t}-1)q_{t}}{g_{t}(q)-\sqrt{\kappa} W_{t}} dW_{t}
+\dfrac{R_{t}^{v}(q_{t}-1)}{(g_{t}(q)-\sqrt{\kappa} W_{t})^{2}}\times
\\&\times\left(\dfrac{\kappa}{2}(q_{t}-1)q_{t}^{2}f''(q_{t})
+q_{t}((\kappa-2)q_{t}-2)f'(q_{t})-2v(q_{t}-1)f(q_{t})\right)dt.
\end{split}
\end{equation*}
It follows that $(M_{t})_{t\geq 0}$ is a local martingale (hence a true one, $f$ being bounded) if and only if
\begin{displaymath}
\dfrac{\kappa}{2}(q_{t}-1)q_{t}^{2}f''(q_{t})
+q_{t}((\kappa-2)q_{t}-2)f'(q_{t})
-2v(q_{t}-1)f(q_{t})\equiv 0,
\end{displaymath}
which gives the equation \eqref{EqDiffSLEConfRestr}.
\end{proof}

\eqref{EqDiffSLEConfRestr} is the differential equation for 
$\bar{p}_{\kappa,v}$. However, we do not know \textit{a priori} that 
$\bar{p}_{\kappa,v}$ is $\mathcal{C}^{2}$-regular. The idea is to show that both $\bar{p}_{\kappa,v}$ and a solution of \eqref{EqDiffSLEConfRestr} with right boundary conditions are fixed points of a contracting operator, and thus coincide. We will do this for the case $\kappa=4$ which interests us.

\begin{proposition}
\label{PropConvP}
Let $q>1$, $v>0$. 
\begin{displaymath}
\lim_{n\rightarrow +\infty}p^{\widetilde{\mathtt{H}}_{n}}_{1/2,u_{0}(1/2),v}(q)=
p^{\mathbb{H}}_{1/2,u_{0}(1/2),v}(q)=1-q^{-\sqrt{v}}.
\end{displaymath}
\end{proposition}

\begin{proof}
By definition
\begin{displaymath}
p^{\mathbb{H}}_{1/2,u_{0}(1/2),v}(q)=1-\bar{p}_{4,v}(q).
\end{displaymath}
According to Proposition \ref{PropProbMetrGraph},
\begin{displaymath}
\lim_{n\rightarrow +\infty}p^{\widetilde{\mathtt{H}}_{n}}_{1/2,u_{0}(1/2),v}(q)=
1-q^{-2\sqrt{u_{0}(1/2)v}}=1-q^{-\sqrt{v}}.
\end{displaymath}

Let $f_{v}(q):=q^{-\sqrt{v}}$. With $\kappa=4$, the ODE \eqref{EqDiffSLEConfRestr} becomes
\begin{displaymath}
f''+\dfrac{1}{q}f'-\dfrac{v}{q^{2}}f=0
\end{displaymath}
and it is satisfied by $f_{v}$. According to Lemma \ref{LemDiffEq}, $(R_{t}^{v}f_{v}(q_{t}))_{t\geq 0}$ is a martingale (we use the notations \eqref{NotatRtqt} and $\kappa=4$) for any initial value of $q_{0}$. 
In particular for any $t>0$
\begin{displaymath}
f_{v}(q_{0})=\mathbb{E}[R_{t}^{v}f_{v}(q_{t})].
\end{displaymath}
The same is true if we replace $f_{v}$ by $\bar{p}_{4,v}$ (\eqref{IntegralEq}). Thus,
\begin{equation}
\label{Contraction}
f_{v}(q_{0})-\bar{p}_{4,v}(q_{0})=\mathbb{E}[R_{t}^{v}(f_{v}(q_{t})-\bar{p}_{4,v}(q_{t}))]
\end{equation}
for any starting value of $q_{0}\in(1,+\infty)$ and $t>0$.

$\bar{p}_{4,v}$ is non-increasing on $(1,+\infty)$ with boundary limits 
\begin{displaymath}
\bar{p}_{4,v}(1)=1,\qquad 
\bar{p}_{4,v}(+\infty)=0.
\end{displaymath}
Moreover $\bar{p}_{4,v}$ is continuous. Indeed, let $q\in (1,+\infty)$. A.s.\ there is no excursion in $\mathcal{E}^{\mathbb{H}}_{v}([1,q])$ with endpoint $(q,0)$. This means that $\bar{p}_{4,v}$ is left-continuous at $q$.
Moreover, a.s.\ there is $\varepsilon>0$ such that there is no excursion in $\mathcal{E}^{\mathbb{H}}_{v}([1,q+\varepsilon)$ with an endpoint in $[q,q+\varepsilon)\times\lbrace 0\rbrace$ that intersects an independent $\hbox{SLE}_{4}$ curve.
This implies that $\bar{p}_{4,v}$ is right-continuous at $q$. From the continuity of $\bar{p}_{4,v}$ 
follows that there is $\hat{q}\in (1,+\infty)$ such that
\begin{displaymath}
\vert f_{v}(\hat{q})-\bar{p}_{4,v}(\hat{q})\vert = \max_{q\in (1,+\infty)} \vert f_{v}(q)-\bar{p}_{4,v}(q)\vert.
\end{displaymath}
Let $t>0$ and let $\hat{q}$ be the initial value $q_{0}$ of $(q_{s})_{s\geq 0}$. From
\eqref{Contraction} we get that
\begin{displaymath}
\vert f_{v}(\hat{q})-\bar{p}_{4,v}(\hat{q})\vert\leq\mathbb{E}[R_{t}^{v}]
\vert f_{v}(\hat{q})-\bar{p}_{4,v}(\hat{q})\vert.
\end{displaymath} 
But a.s.\ $R_{t}<1$ and $\mathbb{E}[R_{t}^{v}]<1$. This implies that
\begin{displaymath}
\vert f_{v}(\hat{q})-\bar{p}_{4,v}(\hat{q})\vert = \max_{q\in (1,+\infty)} \vert f_{v}(q)-\bar{p}_{4,v}(q)\vert=0
\end{displaymath}
and that
\begin{displaymath}
\bar{p}_{4,v}(q)\equiv q^{-\sqrt{v}}.
\qedhere
\end{displaymath}
\end{proof}

\section{Convergence to CLE}
\label{secConv}

In this section we prove the convergence results.

Let $Q_{l}:=(-l,l)\times(0,l)$. Let $\mathcal{L}_{\alpha}^{\mathtt{H}_{n}\cap Q_{l},T}$  be the loops in
$\mathcal{L}_{\alpha}^{\mathtt{H}_{n}}$ that are contained in $Q_{l}$ and do at least $T$ jumps. Let
$\mathcal{L}_{\alpha}^{Q_{l}}$ be the Brownian loops in $\mathcal{L}_{\alpha}^{\mathbb{H}}$ that are contained
in $Q_{l}$. From \cite{BrugCamiaLis2014RWLoopsCLE} follows that for $\alpha\in (0,1/2]$, $l>0$ and $\theta\in(16/9,2)$, $\mathcal{F}_{\rm ext}(\mathcal{L}_{\alpha}^{\mathtt{H}_{n}\cap Q_{l},n^{\theta}})$ converges in law to $\mathcal{F}_{\rm ext}(\mathcal{L}_{\alpha}^{Q_{l}})$.

\begin{lemma}
\label{LemConvCutOffHP}
Let $\alpha\in (0,1/2]$ and $\theta\in(16/9,2)$.
$\mathcal{F}_{\rm ext}(\mathcal{L}_{\alpha}^{\mathtt{H}_{n},n^{\theta}})$ converges in law to 
$\mathcal{F}_{\rm ext}(\mathcal{L}_{\alpha}^{\mathbb{H}})$.
\end{lemma}

\begin{proof}
Let $z_{1},\dots,z_{j}\in\mathbb{H}$. To deduce that 
$\mathcal{F}_{\rm ext}(\mathcal{L}_{\alpha}^{\mathtt{H}_{n},n^{\theta}})[z_{1},\dots,z_{j}]$ converges in law to
$\mathcal{F}_{\rm ext}(\mathcal{L}_{\alpha}^{\mathbb{H}})[z_{1},\dots,z_{j}]$ from the result of \cite{BrugCamiaLis2014RWLoopsCLE} we need only to show that
\begin{displaymath}
\lim_{l\rightarrow +\infty}\liminf_{n\rightarrow +\infty}\mathbb{P}(
\text{Contours of}~\mathcal{F}_{\rm ext}(\mathcal{L}_{\alpha}^{\mathtt{H}_{n},n^{\theta}})[z_{1},\dots,z_{j}]
~\text{contained in}~Q_{l})=1.
\end{displaymath}

Let $\varepsilon\in(0,1/2)$. There is $l_{0}>0$ such that
\begin{displaymath}
\mathbb{P}\left(
\text{Contours of}~\mathcal{F}_{\rm ext}(\mathcal{L}_{\alpha}^{\mathbb{H}})[z_{1},\dots,z_{j}]
~\text{contained in}~Q_{l_{0}}\right)\geq 1-\varepsilon.
\end{displaymath}
Denote
\begin{displaymath}
\partial_{\mathbb{H}}Q_{l}:=(\lbrace -l\rbrace\times(0,l])\cup(\lbrace l\rbrace\times(0,l])
\cup([-l,l]\times\lbrace l\rbrace).
\end{displaymath}
There is $l_{1}>l_{0}$ such that
\begin{displaymath}
\mathbb{P}(\exists\gamma\in\mathcal{L}_{\alpha}^{\mathbb{H}}, \gamma\cap Q_{l_{0}}\neq\emptyset,
\gamma\cap\partial_{\mathbb{H}}Q_{l_{1}}\neq\emptyset)\leq\varepsilon.
\end{displaymath}
Then
\begin{equation*}
\begin{split}
\lim_{n\rightarrow +\infty}\mathbb{P}(&\text{Contours of}~
\mathcal{F}_{\rm ext}(\mathcal{L}_{\alpha}^{\mathtt{H}_{n}\cap Q_{l_{1}},n^{\theta}})
[z_{1},\dots,z_{j}]~\text{contained in}~Q_{l_{0}})
\\&=\mathbb{P}(\text{Contours of}~\mathcal{F}_{\rm ext}(\mathcal{L}_{\alpha}^{Q_{l_{1}}})
[z_{1},\dots,z_{j}]~\text{contained in}~Q_{l_{0}})
\\&\geq\mathbb{P}\left(
\text{Contours of}~\mathcal{F}_{\rm ext}(\mathcal{L}_{\alpha}^{\mathbb{H}})[z_{1},\dots,z_{j}]
~\text{contained in}~Q_{l_{0}}\right)\geq 1-\varepsilon.
\end{split}
\end{equation*}
According to the approximation of \cite{LawlerFerreras2007RWLoopSoup},
\begin{multline*}
\lim_{n\rightarrow +\infty}\mathbb{P}(\exists\gamma\in\mathcal{L}_{\alpha}^{\mathtt{H}_{n},n^{\theta}}, 
\gamma\cap Q_{l_{0}}\neq\emptyset,\gamma\cap\partial_{\mathbb{H}}Q_{l_{1}}\neq\emptyset)
\\=\mathbb{P}(\exists\gamma\in\mathcal{L}_{\alpha}^{\mathbb{H}}, \gamma\cap Q_{l_{0}}\neq\emptyset,
\gamma\cap\partial_{\mathbb{H}}Q_{l_{1}}\neq\emptyset)\leq\varepsilon.
\end{multline*}
But
\begin{equation*}
\begin{split}
\mathbb{P}(\text{Contours of}~&
\mathcal{F}_{\rm ext}(\mathcal{L}_{\alpha}^{\mathtt{H}_{n},n^{\theta}})[z_{1},\dots,z_{j}]
~\text{contained in}~Q_{l_{0}})\geq
\\&\mathbb{P}(\text{Contours of}~\mathcal{F}_{\rm ext}(\mathcal{L}_{\alpha}^{\mathtt{H}_{n}\cap Q_{l_{1}},n^{\theta}})[z_{1},\dots,z_{j}]~\text{contained in}~Q_{l_{0}})\\
&-\mathbb{P}(\exists\gamma\in\mathcal{L}_{\alpha}^{\mathtt{H}_{n},n^{\theta}}, 
\gamma\cap Q_{l_{0}}\neq\emptyset,\gamma\cap\partial_{\mathbb{H}}Q_{l_{1}}\neq\emptyset).
\end{split}
\end{equation*}
Thus,
\begin{displaymath}
\liminf_{n\rightarrow +\infty}\mathbb{P}(\text{Contours of}~\mathcal{F}_{\rm ext}(\mathcal{L}_{\alpha}^{\mathtt{H}_{n},n^{\theta}})[z_{1},\dots,z_{j}]~\text{contained in}
~Q_{l_{0}})\geq 1-2\varepsilon.
\qedhere
\end{displaymath}
\end{proof}

From now on $\theta\in(16/9,2)$ will be fixed. $\alpha$ will belong to $(0,1/2]$. For 
$z_{0}\in\mathbb{H}$, we define
\begin{displaymath}
\delta_{\alpha,n}(z_{0}):=\max\lbrace 
d(z,\mathcal{F}_{\rm ext}(\mathcal{L}_{\alpha}^{\mathtt{H}_{n},n^{\theta}})(z_{0}))\vert
z\in
\mathcal{F}_{\rm ext}(\mathcal{L}_{\alpha}^{\widetilde{\mathtt{H}}_{n}})(z_{0})\rbrace.
\end{displaymath}
By $z\in\mathcal{F}_{\rm ext}(\mathcal{L}_{\alpha}^{\widetilde{\mathtt{H}}_{n}})(z_{0})$ we mean that $z$ is a point on the contour $\mathcal{F}_{\rm ext}(\mathcal{L}_{\alpha}^{\widetilde{\mathtt{H}}_{n}})(z_{0})$. The random variable $\delta_{\alpha,n}(z_{0})$ is defined only when 
$\mathcal{F}_{\rm ext}(\mathcal{L}_{\alpha}^{\mathtt{H}_{n},n^{\theta}})(z_{0})$ is defined, which happens with probability converging to $1$.

\begin{lemma}
\label{LemPosDist}
Assume that $\mathcal{F}_{\rm ext}(\mathcal{L}_{\alpha}^{\widetilde{\mathtt{H}}_{n}})$ does not converge in law to
$\mathcal{F}_{\rm ext}(\mathcal{L}_{\alpha}^{\mathbb{H}})$. Then there is $z_{\alpha,0}\in\mathbb{H}$ such that
$\delta_{\alpha,n}(z_{\alpha,0})$ does not converge in law to $0$.
\end{lemma}

\begin{proof}
If $\mathcal{F}_{\rm ext}(\mathcal{L}_{\alpha}^{\widetilde{\mathtt{H}}_{n}})$ does not converge in law to
$\mathcal{F}_{\rm ext}(\mathcal{L}_{\alpha}^{\mathbb{H}})$ then by definition there are 
$z_{1},\dots,z_{j}\in\mathbb{H}$ such that 
$\mathcal{F}_{\rm ext}(\mathcal{L}_{\alpha}^{\widetilde{\mathtt{H}}_{n}})[z_{1},\dots,z_{j}]$ does not converge in law to
\\$\mathcal{F}_{\rm ext}(\mathcal{L}_{\alpha}^{\mathbb{H}})[z_{1},\dots,z_{j}]$. To the contrary
$\mathcal{F}_{\rm ext}(\mathcal{L}_{\alpha}^{\mathtt{H}_{n},n^{\theta}})[z_{1},\dots,z_{j}]$ does converge in law to
$\mathcal{F}_{\rm ext}(\mathcal{L}_{\alpha}^{\mathbb{H}})[z_{1},\dots,z_{j}]$. Since each contour of 
$\mathcal{F}_{\rm ext}(\mathcal{L}_{\alpha}^{\mathtt{H}_{n},n^{\theta}})[z_{1},\dots,z_{j}]$ is surrounded by a contour of
$\mathcal{F}_{\rm ext}(\mathcal{L}_{\alpha}^{\widetilde{\mathtt{H}}_{n}})[z_{1},\dots,z_{j}]$, one of 
$\delta_{\alpha,n}(z_{i})$ must not converge in law to $0$.
\end{proof}

Let $z_{\alpha,0}$ be defined by the previous lemma under the non-convergence assumption. 
The set of points $z$ on the metric graph contained in or surrounded by
$\mathcal{F}_{\rm ext}(\mathcal{L}_{\alpha}^{\widetilde{\mathtt{H}}_{n}})(z_{\alpha,0})$
and not in the interior surrounded by
$\mathcal{F}_{\rm ext}(\mathcal{L}_{\alpha}^{\mathtt{H}_{n},n^{\theta}})(z_{\alpha,0})$, such that 
$d(z,\mathcal{F}_{\rm ext}(\mathcal{L}_{\alpha}^{\mathtt{H}_{n},n^{\theta}})(z_{\alpha,0}))=\delta_{\alpha,n}(z_{\alpha,0})\wedge 1$,
is non-empty (when $\delta_{\alpha,n}(z_{\alpha,0})$ is defined).
Indeed, $\mathcal{F}_{\rm ext}(\mathcal{L}_{\alpha}^{\widetilde{\mathtt{H}}_{n}})(z_{\alpha,0})$ plus the set of points it surrounds, minus the interior surrounded by
$\mathcal{F}_{\rm ext}(\mathcal{L}_{\alpha}^{\mathtt{H}_{n},n^{\theta}})(z_{\alpha,0})$,
 is connected and compact. Let $Z_{\alpha,n}$ be a random point taking values in the above set, for instance the maximum for the lexicographical order.

\begin{lemma}
\label{LemConvDistPoint}
Assume that $\mathcal{F}_{\rm ext}(\mathcal{L}_{\alpha}^{\widetilde{\mathtt{H}}_{n}})$ does not converge in law to
$\mathcal{F}_{\rm ext}(\mathcal{L}_{\alpha}^{\mathbb{H}})$. Then there is a sub-sequence of indices $n_{\alpha,0}$ such that
the joint law of
\begin{displaymath}
(\mathcal{F}_{\rm ext}(\mathcal{L}_{\alpha}^{\mathtt{H}_{n_{\alpha,0}},n_{\alpha,0}^{\theta}})(z_{\alpha,0}),
Z_{\alpha,n_{\alpha,0}})
\end{displaymath}
has a limit when $n_{\alpha,0}\rightarrow +\infty$. It is a law on
\begin{displaymath}
(\mathcal{F}_{\rm ext}(\mathcal{L}_{\alpha}^{\mathbb{H}})(z_{\alpha,0}),Z_{\alpha})
\end{displaymath}
satisfying the property that with positive probability the point $Z_{\alpha}$ is not contained or surrounded by
$\mathcal{F}_{\rm ext}(\mathcal{L}_{\alpha}^{\mathbb{H}})(z_{\alpha,0})$.
\end{lemma}

\begin{proof}
$\delta_{\alpha,n}(z_{\alpha,0})$ does not converge in law to $0$. This means that there is $\varepsilon>0$ and a sub-sequence of indices $n'$ such that
\begin{equation}
\label{DistPteps}
\forall n', \mathbb{P}(d(Z_{\alpha,n'},\mathcal{F}_{\rm ext}(\mathcal{L}_{\alpha}^
{\mathtt{H}_{n'},n'^{\theta}})(z_{\alpha,0}))\geq\varepsilon)\geq\varepsilon.
\end{equation}
The sub-sequence of random variables
\begin{displaymath}
(\mathcal{F}_{\rm ext}(\mathcal{L}_{\alpha}^{\mathtt{H}_{n',n'^{\theta}}})(z_{\alpha,0}),
Z_{\alpha,n'})
\end{displaymath}
is tight. Indeed the first component of the couple converges in law and the second is by definition at distance at most $1$ from the first. Thus there is a sub-sequence of indices $n_{\alpha,0}$ out of $n'$ such that there is a convergence in law. 
$\mathcal{F}_{\rm ext}(\mathcal{L}_{\alpha}^
{\mathtt{H}_{n_{\alpha,0}},n_{\alpha,0}^{\theta}})(z_{\alpha,0})$ converges in law 
$\mathcal{F}_{\rm ext}(\mathcal{L}_{\alpha}
^{\mathbb{H}})(z_{\alpha,0})$. 
Let $Z_{\alpha}$ be defined as the second component of the limit in law of
$(\mathcal{F}_{\rm ext}(\mathcal{L}_{\alpha}^
{\mathtt{H}_{n_{\alpha,0}},n_{\alpha,0}^{\theta}})(z_{\alpha,0}),
Z_{\alpha,n_{\alpha,0}})$.
\eqref{DistPteps} implies that
\begin{displaymath}
\mathbb{P}(d(Z_{\alpha},\mathcal{F}_{\rm ext}(\mathcal{L}_{\alpha}^{\mathbb{H}})
(z_{\alpha,0}))\geq\varepsilon)\geq\varepsilon.
\end{displaymath}
Moreover, a.s.\ $Z_{\alpha}$ cannot be in the interior surrounded by 
$\mathcal{F}_{\rm ext}(\mathcal{L}_{\alpha}^{\mathbb{H}})(z_{\alpha,0})$ because $Z_{\alpha,n}$ is not surrounded by
$\mathcal{F}_{\rm ext}(\mathcal{L}_{\alpha}^{\mathtt{H}_{n},n^{\theta}})(z_{\alpha,0})$.
\end{proof}

From now on $(z_{j})_{j\geq 1}$ will be a fixed everywhere dense sequence in $\mathbb{H}$.

\begin{lemma}
\label{LemSeqSubSeq}
Assume that $\mathcal{F}_{\rm ext}(\mathcal{L}_{\alpha}^{\widetilde{\mathtt{H}}_{n}})$ does not converge in law to
$\mathcal{F}_{\rm ext}(\mathcal{L}_{\alpha}^{\mathbb{H}})$. Then there is a family of sub-sequences of indices $n_{\alpha,j}$ such that
\begin{itemize}
\item $n_{\alpha,0}$ is given by Lemma \ref{LemConvDistPoint}.
\item $n_{\alpha,j+1}$ is a sub-sequence of $n_{\alpha,j}$.
\item The random variable
\begin{displaymath}
(\mathcal{F}_{\rm ext}(\mathcal{L}_{\alpha}^{\mathtt{H}_{n_{\alpha,j}},n_{\alpha,j}^{\theta}})
[z_{\alpha,0},z_{1},\dots,z_{j}],Z_{\alpha,n_{\alpha,j}})
\end{displaymath}
converges in law as $n_{\alpha,j}\rightarrow +\infty$ and the limit defines the joint law of
\begin{displaymath}
(\mathcal{F}_{\rm ext}(\mathcal{L}_{\alpha}^{\mathbb{H}})
[z_{\alpha,0},z_{1},\dots,z_{j}],Z_{\alpha}).
\end{displaymath}
\item The family of joint laws on $(\mathcal{F}_{\rm ext}(\mathcal{L}_{\alpha}^{\mathbb{H}})
[z_{\alpha,0},z_{1},\dots,z_{j}],Z_{\alpha})_{j\geq 1}$ is consistent in the sense that the law on 
$(\mathcal{F}_{\rm ext}(\mathcal{L}_{\alpha}^{\mathbb{H}})
[z_{\alpha,0},z_{1},\dots,z_{j}],Z_{\alpha})$ induced by the law of
\\$(\mathcal{F}_{\rm ext}(\mathcal{L}_{\alpha}^{\mathbb{H}})
[z_{\alpha,0},z_{1},\dots,z_{j+1}],Z_{\alpha})$ is the same as the one given by the convergence. In particular the law on
$(\mathcal{F}_{\rm ext}(\mathcal{L}_{\alpha}^{\mathbb{H}})(z_{\alpha,0}),Z_{\alpha})$ is the one given by Lemma
\ref{LemConvDistPoint}.
\item The family of laws of $(\mathcal{F}_{\rm ext}(\mathcal{L}_{\alpha}^{\mathbb{H}})
[z_{\alpha,0},z_{1},\dots,z_{j}],Z_{\alpha})_{j\geq 1}$ uniquely defines a law on
$(\mathcal{F}_{\rm ext}(\mathcal{L}_{\alpha}^{\mathbb{H}}),Z_{\alpha})$.
\end{itemize}
\end{lemma}

\begin{proof}
The consistency of law follows from the fact that $n_{\alpha,j+1}$ is a sub-sequence of $n_{\alpha,j}$. A contour loop in $\mathcal{F}_{\rm ext}(\mathcal{L}_{\alpha}^{\mathbb{H}})$ almost surely surrounds one of the $z_{j}$ points. Thus the fact that a consistent family of laws on $(\mathcal{F}_{\rm ext}(\mathcal{L}_{\alpha}^{\mathbb{H}})
[z_{\alpha,0},z_{1},\dots,z_{j}],Z_{\alpha})_{j\geq 1}$ uniquely defines a law on
$(\mathcal{F}_{\rm ext}(\mathcal{L}_{\alpha}^{\mathbb{H}}),Z_{\alpha})$ follows from the Kolmogorov extension theorem.

Next we explain how we extract $n_{\alpha,j+1}$ out of $n_{\alpha,j}$. By construction,
the sub-sequence \\$(\mathcal{F}_{\rm ext}(\mathcal{L}_{\alpha}^{\mathtt{H}_{n_{\alpha,j}},n_{\alpha,j}^{\theta}})
[z_{\alpha,0},z_{1},\dots,z_{j}],Z_{\alpha,n_{\alpha,j}})$ converges in law as $n_{\alpha,j}\rightarrow +\infty$ and defines a joint law on $(\mathcal{F}_{\rm ext}(\mathcal{L}_{\alpha}^{\mathbb{H}})
[z_{\alpha,0},z_{1},\dots,z_{j}],Z_{\alpha})$. Moreover we have the convergence in law of 
$\mathcal{F}_{\rm ext}(\mathcal{L}_{\alpha}^{\mathtt{H}_{n_{\alpha,j}},n_{\alpha,j}^{\theta}})(z_{j+1})$ to
$\mathcal{F}_{\rm ext}(\mathcal{L}_{\alpha}^{\mathbb{H}})(z_{j+1})$. Thus the sub-sequence
$(\mathcal{F}_{\rm ext}(\mathcal{L}_{\alpha}^{\mathtt{H}_{n_{\alpha,j}},n_{\alpha,j}^{\theta}})
[z_{\alpha,0},z_{1},\dots,z_{j+1}],Z_{\alpha,n_{\alpha,j}})$ is tight and one can extract a subset of indices $n_{\alpha,j+1}$ such that it converges in law. The limit law is a law on
\\$(\mathcal{F}_{\rm ext}(\mathcal{L}_{\alpha}^{\mathbb{H}})
[z_{\alpha,0},z_{1},\dots,z_{j+1}],Z_{\alpha})$.
\end{proof}

\begin{thm}
\label{ThmConvOneHalf}
$\mathcal{F}_{\rm ext}(\mathcal{L}_{1/2}^{\mathtt{H}_{n}})$ and 
$\mathcal{F}_{\rm ext}(\mathcal{L}_{1/2}^{\widetilde{\mathtt{H}}_{n}})$ converge in law as
$n\rightarrow +\infty$ to $\mathcal{F}_{\rm ext}(\mathcal{L}_{1/2}^{\mathbb{H}})$, that is to say to a $\hbox{CLE}_{4}$ on
$\mathbb{H}$.
\end{thm}

\begin{proof}
It is enough to prove the convergence of 
$\mathcal{F}_{\rm ext}(\mathcal{L}_{1/2}^{\widetilde{\mathtt{H}}_{n}})$. Indeed we already have the convergence for 
$\mathcal{F}_{\rm ext}(\mathcal{L}_{1/2}^{\mathtt{H}_{n},n^{\theta}})$ and each contour 
$\mathcal{F}_{\rm ext}(\mathcal{L}_{1/2}^{\mathtt{H}_{n}})(z)$ lies between the contour
$\mathcal{F}_{\rm ext}(\mathcal{L}_{1/2}^{\mathtt{H}_{n},n^{\theta}})(z)$ and the contour
$\mathcal{F}_{\rm ext}(\mathcal{L}_{1/2}^{\widetilde{\mathtt{H}}_{n}})(z)$.  

Assume that $\mathcal{F}_{\rm ext}(\mathcal{L}_{1/2}^{\widetilde{\mathtt{H}}_{n}})$ does not converge in law to
$\mathcal{F}_{\rm ext}(\mathcal{L}_{1/2}^{\mathbb{H}})$. Let $z_{1/2,0}$ be the point defined by Lemma
\ref{LemPosDist} and $n_{1/2,j}$ the sub-sequences defined by Lemma \ref{LemSeqSubSeq}. We also consider the joint law of $(\mathcal{F}_{\rm ext}(\mathcal{L}_{1/2}^{\mathbb{H}}),Z_{1/2})$ defined by Lemma \ref{LemSeqSubSeq}.

For $u,v>0$ and $q>1$ we consider additional independent Poisson point processes of excursions
$\mathcal{E}^{\mathbb{H}}_{u}((-\infty,0])$ and $\mathcal{E}^{\mathbb{H}}_{v}([1,q])$.
Let $A_{1/2,u,v}(q)$ be the event that is satisfied if either 
an excursion from $\mathcal{E}^{\mathbb{H}}_{u}((-\infty,0])$ and one from 
$\mathcal{E}^{\mathbb{H}}_{v}([1,q])$ intersect each other or both intersect a common contour from
$\mathcal{F}_{\rm ext}(\mathcal{L}_{1/2}^{\mathbb{H}})$. By definition
\begin{displaymath}
\mathbb{P}(A_{1/2,u,v}(q))=p^{\mathbb{H}}_{1/2,u,v}(q).
\end{displaymath}

Let $A^{+}_{1/2,u,v}(q)$ be the event that is satisfied if one of the following conditions holds:
\begin{itemize}
\item An excursion from $\mathcal{E}^{\mathbb{H}}_{u}((-\infty,0])$ and one from 
$\mathcal{E}^{\mathbb{H}}_{v}([1,q])$ intersect each other.
\item An excursion from $\mathcal{E}^{\mathbb{H}}_{u}((-\infty,0])$ and one from 
$\mathcal{E}^{\mathbb{H}}_{v}([1,q])$ intersect a common contour from
$\mathcal{F}_{\rm ext}(\mathcal{L}_{1/2}^{\mathbb{H}})$.
\item An excursion from $\mathcal{E}^{\mathbb{H}}_{u}((-\infty,0])$ intersects
$\mathcal{F}_{\rm ext}(\mathcal{L}_{1/2}^{\mathbb{H}})(z_{1/2,0})$
and an excursion from $\mathcal{E}^{\mathbb{H}}_{v}([1,q])$ hits or surrounds
$Z_{1/2}$.
\item An excursion from $\mathcal{E}^{\mathbb{H}}_{v}([1,q])$ intersects
$\mathcal{F}_{\rm ext}(\mathcal{L}_{1/2}^{\mathbb{H}})(z_{1/2,0})$
and an excursion from $\mathcal{E}^{\mathbb{H}}_{u}((-\infty,0])$ hits or surrounds
$Z_{1/2}$.
\end{itemize}
We claim that
\begin{displaymath}
\mathbb{P}(A^{+}_{1/2,u,v}(q))>\mathbb{P}(A_{1/2,u,v}(q))
= p^{\mathbb{H}}_{1/2,u,v}(q).
\end{displaymath}
To see that the strict inequity holds, consider the following:
\begin{itemize}
\item Restrict to the event when $Z_{1/2}$ is not contained or surrounded by
the contour $\mathcal{F}_{\rm ext}(\mathcal{L}_{1/2}^{\mathbb{H}})(z_{1/2,0})$, which
has a positive probability.
\item Let $K$ by a compact subset of $\lbrace\Im(z)\geq 0\rbrace$ that contains
$\mathcal{F}_{\rm ext}(\mathcal{L}_{1/2}^{\mathbb{H}})(z_{1/2,0})$ and 
$Z_{1/2}$, such that $\mathbb{H}\setminus K$ is simply connected and such that
$K$ intersects the real line on $(0,+\infty)$ only.
\item Since
$\mathcal{E}^{\mathbb{H}}_{u}((-\infty,0])$ is independent from 
$(\mathcal{F}_{\rm ext}(\mathcal{L}_{1/2}^{\mathbb{H}}),Z_{1/2},K)$, there is a positive 
probability that no excursions in 
$\mathcal{E}^{\mathbb{H}}_{u}((-\infty,0])$, except one, hits K, 
and one excursion hits the contour
$\mathcal{F}_{\rm ext}(\mathcal{L}_{1/2}^{\mathbb{H}})(z_{1/2,0})$
without surrounding $Z_{1/2}$.
Then the point $Z_{1/2}$ is to the right from the region
defined by $\mathcal{E}^{\mathbb{H}}_{u}((-\infty,0])$
and the contours in 
$\mathcal{F}_{\rm ext}(\mathcal{L}_{1/2}^{\mathbb{H}})$ it intersects.
See Figure \ref{CEInterface} again for a representation of this region.
\item Since $\mathcal{E}^{\mathbb{H}}_{v}([1,q])$ is independent from
$(\mathcal{F}_{\rm ext}(\mathcal{L}_{1/2}^{\mathbb{H}}),Z_{1/2},
\mathcal{E}^{\mathbb{H}}_{u}((-\infty,0]))$, there is a positive probability that no excursion from $\mathcal{E}^{\mathbb{H}}_{v}([1,q])$
hits the region defined by $\mathcal{E}^{\mathbb{H}}_{u}((-\infty,0]))$ and the contours in 
$\mathcal{F}_{\rm ext}(\mathcal{L}_{1/2}^{\mathbb{H}})$ intersected by
$\mathcal{E}^{\mathbb{H}}_{u}((-\infty,0]))$, but one excursion
from $\mathcal{E}^{\mathbb{H}}_{v}([1,q])$ surrounds the point
$Z_{1/2}$, which is to the right from this region.
\end{itemize}

See Figure \ref{CEAplus} for the illustration of $A^{+}_{1/2,u,v}(q)\setminus A_{1/2,u,v}(q)$.

\begin{figure}[ht]
\begin{center}
  \includegraphics[scale=0.5]{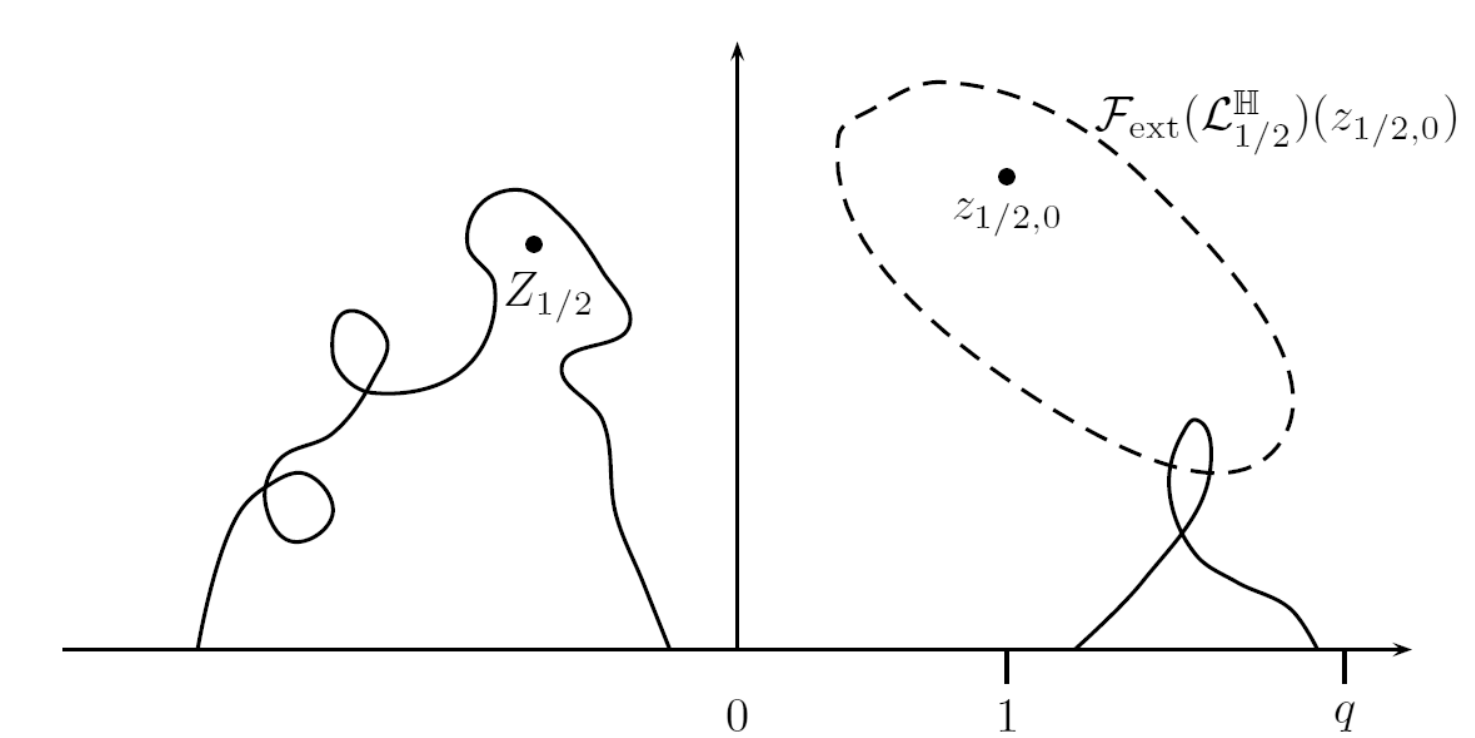}
  \end{center}
  \caption{Illustration of $A^{+}_{1/2,u,v}(q)$ where an excursion from 
  $\mathcal{E}^{\mathbb{H}}_{u}((-\infty,0])$ surrounds $Z_{1/2}$
  and an excursion from $\mathcal{E}^{\mathbb{H}}_{v}([1,q])$
intersects $\mathcal{F}_{\rm ext}(\mathcal{L}_{1/2}^{\mathbb{H}})(z_{1/2,0})$.
  }
  \label{CEAplus}
\end{figure}

Let $j\geq 1$. The events $A_{1/2,u,v}(q,j)$ respectively $A^{+}_{1/2,u,v}(q,j)$ are defined similarly to
$A_{1/2,u,v}(q)$ respectively $A^{+}_{1/2,u,v}(q)$, where the condition of 
$\mathcal{E}^{\mathbb{H}}_{u}((-\infty,0])$ and $\mathcal{E}^{\mathbb{H}}_{v}([1,q])$ intersecting a common contour of
$\mathcal{F}_{\rm ext}(\mathcal{L}_{1/2}^{\mathbb{H}})$ is replaced by the condition of intersecting a common contour of $\mathcal{F}_{\rm ext}(\mathcal{L}_{1/2}^{\mathbb{H}})[z_{1/2,0},z_{1},\dots,z_{j}]$. Then
\begin{displaymath}
\lim_{j\rightarrow +\infty}\mathbb{P}(A_{1/2,u,v}(q,j))=\mathbb{P}(A_{1/2,u,v}(q)),
\quad
\lim_{j\rightarrow +\infty}\mathbb{P}(A^{+}_{1/2,u,v}(q,j))=\mathbb{P}(A^{+}_{1/2,u,v}(q)).
\end{displaymath}

We will denote by  $A^{n}_{1/2,u,v}(q,j)$ and
$A^{n,+}_{1/2,u,v}(q,j)$ the events defined similarly to 
\\$A_{1/2,u,v}(q,j)$ and $A^{+}_{1/2,u,v}(q,j)$ by doing the following replacements:
\begin{itemize}
\item $\mathcal{E}^{\mathbb{H}}_{u}((-\infty,0])$ replaced by 
$\mathcal{E}^{\widetilde{\mathtt{H}}_{n}}_{u}((-\infty,0])$ and
$\mathcal{E}^{\mathbb{H}}_{v}([1,q])$ replaced by 
$\mathcal{E}^{\widetilde{\mathtt{H}}_{n}}_{v}([1,q])$,
\item $Z_{1/2}$ replaced by $Z_{1/2,n}$,
\item $\mathcal{F}_{\rm ext}(\mathcal{L}_{1/2}^{\mathbb{H}})$ replaced by
$\mathcal{F}_{\rm ext}(\mathcal{L}_{1/2}^{\mathtt{H}_{n},n^{\theta}})$ and 
$\mathcal{F}_{\rm ext}(\mathcal{L}_{1/2}^{\mathbb{H}})[z_{1/2,0},z_{1},\dots,z_{j}]$ replaced by
\\$\mathcal{F}_{\rm ext}(\mathcal{L}_{1/2}^{\mathtt{H}_{n},n^{\theta}})[z_{1/2,0},z_{1},\dots,z_{j}]$.
\end{itemize}
$\mathcal{F}_{\rm ext}(\mathcal{L}_{1/2}^{\mathtt{H}_{n},n^{\theta}})
[z_{1/2,0},z_{1},\dots,z_{n}]$ converges in law to $\mathcal{F}_{\rm ext}(\mathcal{L}_{1/2}^{\mathbb{H}})[z_{1/2,0},z_{1},\dots,z_{j}]$, the P.p.p.
$\mathcal{E}^{\widetilde{\mathtt{H}}_{n}}_{u}((-\infty,0])$ to $\mathcal{E}^{\mathbb{H}}_{u}((-\infty,0])$ and
$\mathcal{E}^{\widetilde{\mathtt{H}}_{n}}_{v}([1,q])$ to $\mathcal{E}^{\mathbb{H}}_{v}([1,q])$.
Moreover, in the limit, if an excursion intersects a contour loop in $\mathcal{F}_{\rm ext}(\mathcal{L}_{1/2}^{\mathbb{H}})[z_{1/2,0},z_{1},\dots,z_{j}]$, then a.s.\ it goes inside the interior surrounded by the loop. Thus the intersection still holds for small deformations of the excursion and of the contour. Thus for all $j\geq 1$ we have the convergence
\begin{displaymath}
\lim_{n\rightarrow +\infty}\mathbb{P}(A^{n}_{1/2,u,v}(q,j))=\mathbb{P}(A_{1/2,u,v}(q,j)).
\end{displaymath}
From Lemma \ref{LemSeqSubSeq} follows that
\begin{displaymath}
\lim_{n_{1/2,j}\rightarrow +\infty}\mathbb{P}(A^{n_{1/2,j},+}_{1/2,u,v}(q,j))=
\mathbb{P}(A^{+}_{1/2,u,v}(q,j)).
\end{displaymath}

Each contour of $\mathcal{F}_{\rm ext}(\mathcal{L}_{1/2}^{\mathtt{H}_{n},n^{\theta}})$ is surrounded by a contour of
$\mathcal{F}_{\rm ext}(\mathcal{L}_{1/2}^{\widetilde{\mathtt{H}}_{n}})$ and $Z_{1/2,n}$ belongs to or is surrounded by
$\mathcal{F}_{\rm ext}(\mathcal{L}_{1/2}^{\widetilde{\mathtt{H}}_{n}})(z_{1/2},0)$. Thus, on the event
$A^{n,+}_{1/2,u,v}(q,j)$, an excursion from $\mathcal{E}^{\widetilde{\mathtt{H}}_{n}}_{u}((-\infty,0])$ and one
from $\mathcal{E}^{\widetilde{\mathtt{H}}_{n}}_{v}([1,q])$ either intersect each other or intersect a common contour from
$\mathcal{F}_{\rm ext}(\mathcal{L}_{1/2}^{\widetilde{\mathtt{H}}_{n}})[z_{1/2,0},z_{1},\dots,z_{j}]$. Thus,
\begin{displaymath}
p^{\widetilde{\mathtt{H}}_{n}}_{1/2,u,v}(q)\geq \mathbb{P}(A^{n,+}_{1/2,u,v}(q,j)).
\end{displaymath}
Let $u$ be equal to $u_{0}(1/2)$. Then
\begin{multline*}
p^{\mathbb{H}}_{1/2,u_{0}(1/2),v}(q)=
\lim_{n_{1/2,j}\rightarrow +\infty}
p^{\widetilde{\mathtt{H}}_{n_{1/2,j}}}_{1/2,u_{0}(1/2),v}(q)\geq\\
\lim_{n_{1/2,j}\rightarrow +\infty}\mathbb{P}(A^{n_{1/2,j},+}_{1/2,u_{0}(1/2),v}(q,j))=
\mathbb{P}(A^{+}_{1/2,u_{0}(1/2),v}(q,j)).
\end{multline*}
Taking the limit as $j\rightarrow +\infty$ we get
\begin{multline*}
p^{\mathbb{H}}_{1/2,u_{0}(1/2),v}(q)\geq \lim_{j\rightarrow +\infty}
\mathbb{P}(A^{+}_{1/2,u_{0}(1/2),v}(q,j))=\mathbb{P}(A^{+}_{1/2,u_{0}(1/2),v}(q))
>\\\mathbb{P}(A_{1/2,u_{0}(1/2),v}(q))=p^{\mathbb{H}}_{1/2,u_{0}(1/2),v}(q),
\end{multline*}
which is a contradiction. It follows that $\mathcal{F}_{\rm ext}(\mathcal{L}_{1/2}^{\widetilde{\mathtt{H}}_{n}})$ converges in law to $\mathcal{F}_{\rm ext}(\mathcal{L}_{1/2}^{\mathbb{H}})$.
\end{proof}




\begin{lemma}
\label{LemPositivity}
Let $\alpha\in (0,1/2)$. Let $\bar{\alpha}:=1/2-\alpha$. Let $\mathcal{L}_{\alpha}^{\mathbb{H}}$
and $\mathcal{L}_{\bar{\alpha}}^{\mathbb{H}}$ be independent and let
\begin{displaymath}
\mathcal{L}_{1/2}^{\mathbb{H}}=\mathcal{L}_{\alpha}^{\mathbb{H}}\cup
\mathcal{L}_{\bar{\alpha}}^{\mathbb{H}}.
\end{displaymath}
Let $z\neq \tilde{z}\in\mathbb{H}$.
Let 
$\mathcal{F}_{\rm ext}^{\bullet}(\mathcal{L}_{\bar{\alpha}}^{\mathbb{H}})(\tilde{z})$,
respectively
$\mathcal{F}_{\rm ext}^{\bullet}(\mathcal{L}_{\alpha}^{\mathbb{H}})(z)$,
denote the region surrounded by
$\mathcal{F}_{\rm ext}(\mathcal{L}_{\bar{\alpha}}^{\mathbb{H}})
(\tilde{z})$, 
respectively
$\mathcal{F}_{\rm ext}(\mathcal{L}_{\alpha}^{\mathbb{H}})(z)$,
i.e. the complement in $\mathbb{H}$ 
of the unique unbounded connected component of
$\mathbb{H}\setminus\mathcal{F}_{\rm ext}
(\mathcal{L}_{\bar{\alpha}}^{\mathbb{H}})(\tilde{z})$,
respectively
$\mathbb{H}\setminus\mathcal{F}_{\rm ext}(\mathcal{L}_{\alpha}^{\mathbb{H}})(z)$.
The conditional probability
\begin{displaymath}
\mathbb{P}(\mathcal{F}_{\rm ext}(\mathcal{L}_{1/2}^{\mathbb{H}})(z)\neq
\mathcal{F}_{\rm ext}(\mathcal{L}_{1/2}^{\mathbb{H}})(\tilde{z})\vert 
\mathcal{F}_{\rm ext}(\mathcal{L}_{\alpha}^{\mathbb{H}}),
\mathcal{F}_{\rm ext}(\mathcal{L}_{\bar{\alpha}}^{\mathbb{H}})(\tilde{z}))
\end{displaymath}
is a.s.\ positive on the event
\begin{displaymath}
\mathcal{F}_{\rm ext}^{\bullet}(\mathcal{L}_{\bar{\alpha}}^{\mathbb{H}})(\tilde{z})\cap
\mathcal{F}_{\rm ext}^{\bullet}(\mathcal{L}_{\alpha}^{\mathbb{H}})(z)=\emptyset.
\end{displaymath}
\end{lemma}

\begin{proof}
On the event that 
$\mathcal{F}_{\rm ext}(\mathcal{L}_{\alpha}^{\mathbb{H}})(z)$ 
does not surround $\tilde{z}$ one can choose a continuous path 
$\tilde{\eta}$ joining $\tilde{z}$ to 
$\partial\mathbb{H}=\mathbb{R}\times\lbrace 0\rbrace$ and avoiding
$\mathcal{F}_{\rm ext}(\mathcal{L}_{\alpha}^{\mathbb{H}})(z)$ ($\tilde{\eta}$ is thus random and measurable with respect to
$\mathcal{F}_{\rm ext}(\mathcal{L}_{\alpha}^{\mathbb{H}})(z)$). 
Let $\widetilde{K}$ be the union of 
$\tilde{\eta}$,
$\mathcal{F}_{\rm ext}(\mathcal{L}_{\bar{\alpha}}^{\mathbb{H}})(\tilde{z})$ and
all the contours in $\mathcal{F}_{\rm ext}(\mathcal{L}_{\alpha}^{\mathbb{H}})$ that do intersect either
$\tilde{\eta}$ or
$\mathcal{F}_{\rm ext}(\mathcal{L}_{\bar{\alpha}}^{\mathbb{H}})(\tilde{z})$.
Let $\hbox{Hull}(\widetilde{K})$ be the hull of $\widetilde{K}$, that is to say the complement in $\mathbb{H}$
of the unique unbounded connected component of $\mathbb{H}\setminus\widetilde{K}$.

On the event that 
$\mathcal{F}_{\rm ext}^{\bullet}(\mathcal{L}_{\alpha}^{\mathbb{H}})(z)$ does not intersect
$\mathcal{F}_{\rm ext}^{\bullet}(\mathcal{L}_{\bar{\alpha}}^{\mathbb{H}})(\tilde{z})$, $z$ does not belong to 
$\hbox{Hull}(\widetilde{K})$. One can than choose a path $\eta$ that connects $z$ to $\partial\mathbb{H}$ and avoids 
$\hbox{Hull}(\widetilde{K})$, $\eta$ being random measurable with respect to $\hbox{Hull}(\widetilde{K})$. Let $K$ be the union of 
$\eta$ and all the contours in $\mathcal{F}_{\rm ext}(\mathcal{L}_{\alpha}^{\mathbb{H}})$ that intersects $\eta$. Let 
$\hbox{Hull}(K)$ be the hull of $K$. Figure \ref{CEEnv} is an illustration of $\tilde{\eta}$, $\eta$, $\widetilde{K}$ and $K$.

\begin{figure}[ht]
\begin{center}
  \includegraphics[scale=0.5]{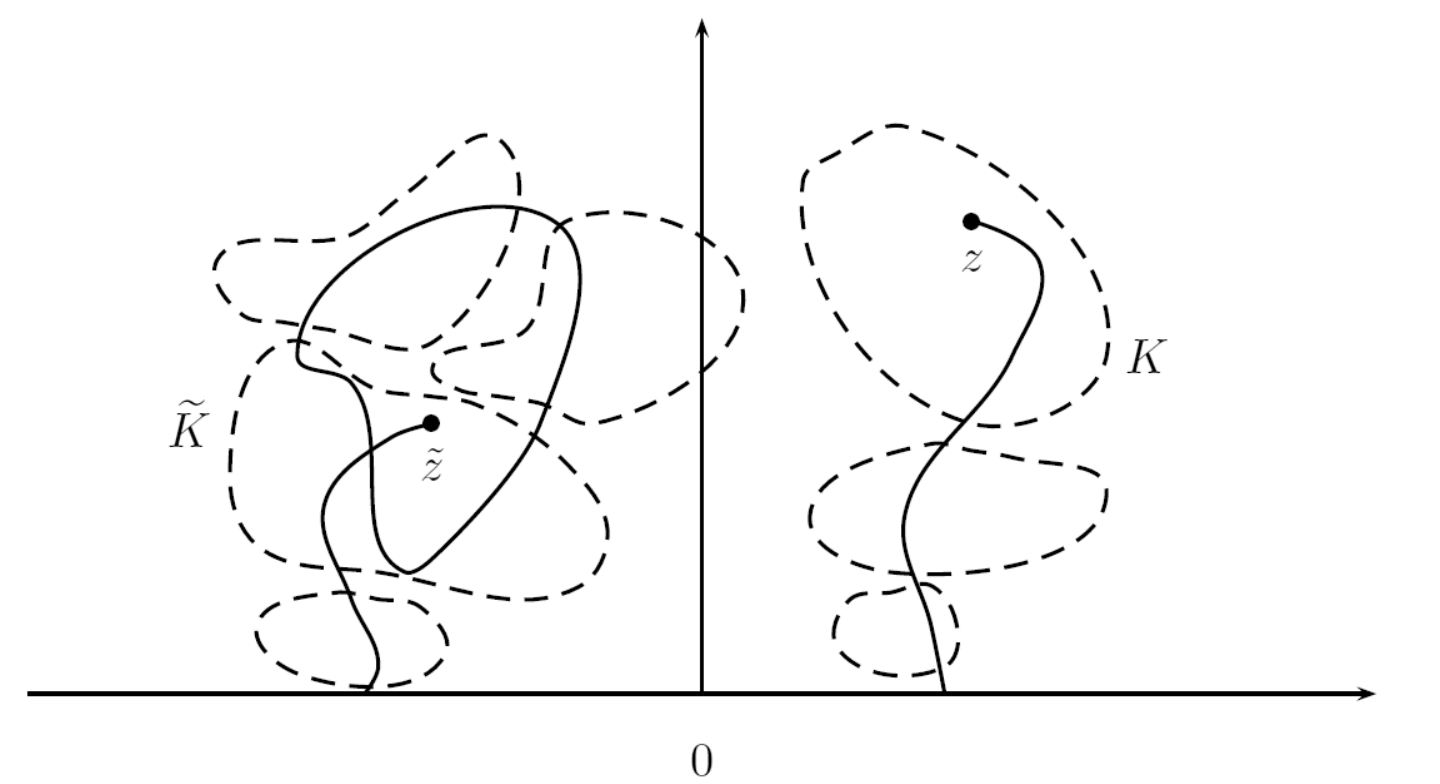}
  \end{center}
  \caption{Illustration of $\tilde{\eta}$, $\eta$, $\widetilde{K}$ and $K$.
$\tilde{\eta}$, $\eta$ and $\mathcal{F}_{\rm ext}(\mathcal{L}_{\bar{\alpha}}^{\mathbb{H}})(\tilde{z})$ 
are drawn in full lines.
Elements of
$\mathcal{F}_{\rm ext}(\mathcal{L}_{\alpha}^{\mathbb{H}})$ 
are drawn in dashed lines.
  }
  \label{CEEnv}
\end{figure}

By construction, on the event 
\begin{displaymath}
\mathcal{F}_{\rm ext}^{\bullet}(\mathcal{L}_{\bar{\alpha}}^{\mathbb{H}})(\tilde{z})\cap
\mathcal{F}_{\rm ext}^{\bullet}(\mathcal{L}_{\alpha}^{\mathbb{H}})(z)=\emptyset,
\end{displaymath}
we have
\begin{itemize}
\item $\hbox{Hull}(K)\cap \hbox{Hull}(\widetilde{K})=\emptyset$,
\item $\mathbb{H}\setminus(\hbox{Hull}(K)\cup \hbox{Hull}(\widetilde{K}))$ is simply connected,
\item no Brownian loop from $\mathcal{L}_{\alpha}^{\mathbb{H}}$ crosses the boundary of 
$\hbox{Hull}(K)$ or $\hbox{Hull}(\widetilde{K})$ and in particular a contour in 
$\mathcal{F}_{\rm ext}(\mathcal{L}_{\alpha}^{\mathbb{H}})$ is either inside $\hbox{Hull}(K)$, $\hbox{Hull}(\widetilde{K})$ or inside the complement $\mathbb{H}\setminus(\hbox{Hull}(K)\cup \hbox{Hull}(\widetilde{K}))$.
\end{itemize}

Conditional on 
\begin{displaymath}
\mathcal{F}_{\rm ext}^{\bullet}(\mathcal{L}_{\bar{\alpha}}^{\mathbb{H}})(\tilde{z})\cap
\mathcal{F}_{\rm ext}^{\bullet}(\mathcal{L}_{\alpha}^{\mathbb{H}})(z)=\emptyset,
\end{displaymath}
and on $\hbox{Hull}(K)$, $\hbox{Hull}(\widetilde{K})$, the law of the contours 
$\mathcal{F}_{\rm ext}(\mathcal{L}_{1/2}^{\mathbb{H}\setminus(\text{Hull}(K)\cup \text{Hull}(\widetilde{K}))})$,
created by the loops $\mathcal{L}_{1/2}^{\mathbb{H}\setminus(\text{Hull}(K)\cup \text{Hull}(\widetilde{K}))}$ from
$\mathcal{L}_{1/2}^{\mathbb{H}}$ that stay inside 
$\mathbb{H}\setminus(\hbox{Hull}(K)\cup \hbox{Hull}(\widetilde{K}))$,
is a $\hbox{CLE}_{4}$ inside
$\mathbb{H}\setminus(\hbox{Hull}(K)\cup \hbox{Hull}(\widetilde{K}))$, and they are conditionally independent from
$\mathcal{L}_{1/2}^{\mathbb{H}}\setminus
\mathcal{L}_{1/2}^{\mathbb{H}\setminus(\text{Hull}(K)\cup \text{Hull}(\widetilde{K}))}$.

Conditional on the event
\begin{displaymath}
\mathcal{F}_{\rm ext}^{\bullet}(\mathcal{L}_{\bar{\alpha}}^{\mathbb{H}})(\tilde{z})\cap
\mathcal{F}_{\rm ext}^{\bullet}(\mathcal{L}_{\alpha}^{\mathbb{H}})(z)=\emptyset
\end{displaymath}
and on $\hbox{Hull}(K)$, $\hbox{Hull}(\widetilde{K})$,
$\mathcal{F}_{\rm ext}(\mathcal{L}_{\alpha}^{\mathbb{H}})$,
$\mathcal{F}_{\rm ext}(\mathcal{L}_{\bar{\alpha}}^{\mathbb{H}})(\tilde{z})$, the probability that
\begin{displaymath}
\mathcal{F}_{\rm ext}(\mathcal{L}_{1/2}^{\mathbb{H}})(z)=
\mathcal{F}_{\rm ext}(\mathcal{L}_{1/2}^{\mathbb{H}})(\tilde{z})
\end{displaymath} 
is less or equal to the probability that $\hbox{Hull}(K)$ and $\hbox{Hull}(\widetilde{K})$ are connected by a cluster of
$\mathcal{L}_{1/2}^{\mathbb{H}}$, which is less or equal to the probability that
given the contours
$\mathcal{F}_{\rm ext}(\mathcal{L}_{1/2}^{\mathbb{H}\setminus(\text{Hull}(K)\cup \text{Hull}(\widetilde{K}))})$ and an independent
loop-soup in $\mathbb{H}$ of parameter $\bar{\alpha}$, there is a contour $\Gamma$ and two loops $\gamma_{1}$ and 
$\gamma_{2}$ in the loop-soup of intensity $\bar{\alpha}$ such that
\begin{itemize}
\item $\gamma_{1}$ intersects $\Gamma$ and $\hbox{Hull}(K)$,
\item $\gamma_{2}$ intersects $\Gamma$ and $\hbox{Hull}(\widetilde{K})$.
\end{itemize}
The latter conditional probability is a.s.\ strictly smaller than $1$. This is what we needed to prove.
\end{proof}

\begin{thm}
\label{ThmConvAll} Let $\alpha\in (0,1/2)$.
$\mathcal{F}_{\rm ext}(\mathcal{L}_{\alpha}^{\mathtt{H}_{n}})$ and 
$\mathcal{F}_{\rm ext}(\mathcal{L}_{\alpha}^{\widetilde{\mathtt{H}}_{n}})$ converge in law as
$n\rightarrow +\infty$ to $\mathcal{F}_{\rm ext}(\mathcal{L}_{\alpha}^{\mathbb{H}})$, that is to say to a 
$\hbox{CLE}_{\kappa(\alpha)}$ on $\mathbb{H}$.
\end{thm}

\begin{proof}
As for Theorem \ref{ThmConvOneHalf}, it is enough to prove that 
$\mathcal{F}_{\rm ext}(\mathcal{L}_{\alpha}^{\widetilde{\mathtt{H}}_{n}})$ converges in law to
$\mathcal{F}_{\rm ext}(\mathcal{L}_{\alpha}^{\mathbb{H}})$. Let us assume that this is not the case.
Let $z_{\alpha,0}$ be the point and $n_{\alpha,0}$ the sub-sequence defined by Lemma \ref{LemPosDist}. 
We also consider the joint law of 
$(\mathcal{F}_{\rm ext}(\mathcal{L}_{\alpha}^{\mathbb{H}}),Z_{\alpha})$ defined by Lemma \ref{LemSeqSubSeq}. 

Since
\begin{displaymath}
\lim_{z\rightarrow\infty}
\mathbb{P}\left(\mathcal{F}_{\rm ext}(\mathcal{L}_{\alpha}^{\mathbb{H}})(z_{\alpha,0})=
\mathcal{F}_{\rm ext}(\mathcal{L}_{\alpha}^{\mathbb{H}})(z)\right)=0
\end{displaymath}
and
\begin{displaymath}
\mathbb{P}\left(d(Z_{\alpha},
\mathcal{F}_{\rm ext}(\mathcal{L}_{\alpha}^{\mathbb{H}})(z_{\alpha,0}))>0\right)>0,
\end{displaymath}
we can choose $\tilde{z}\in\mathbb{H}$ such that
\begin{displaymath}
\mathbb{P}\left(\mathcal{F}_{\rm ext}(\mathcal{L}_{\alpha}^{\mathbb{H}})(z_{\alpha,0})=
\mathcal{F}_{\rm ext}(\mathcal{L}_{\alpha}^{\mathbb{H}})(\tilde{z})\right)<
\mathbb{P}\left(d(Z_{\alpha},
\mathcal{F}_{\rm ext}(\mathcal{L}_{\alpha}^{\mathbb{H}})(z_{\alpha,0}))>0\right).
\end{displaymath}
In that way
\begin{displaymath}
\mathbb{P}\left(d(Z_{\alpha},
\mathcal{F}_{\rm ext}(\mathcal{L}_{\alpha}^{\mathbb{H}})(z_{\alpha,0}))>0,
\mathcal{F}_{\rm ext}(\mathcal{L}_{\alpha}^{\mathbb{H}})(z_{\alpha,0})\neq
\mathcal{F}_{\rm ext}(\mathcal{L}_{\alpha}^{\mathbb{H}})(\tilde{z})
\right)>0.
\end{displaymath}

Let $\bar{\alpha}:=1/2-\alpha$. We take $\mathcal{L}_{\bar{\alpha}}^{\mathbb{H}}$ independent from
$(\mathcal{L}_{\alpha}^{\mathbb{H}},Z_{\alpha})$ and $\mathcal{L}_{\bar{\alpha}}^{\widetilde{\mathtt{H}}_{n}}$ independent from $(\mathcal{L}_{\alpha}^{\widetilde{\mathtt{H}}_{n}},Z_{\alpha,n})$. We define
$\mathcal{L}_{1/2}^{\mathbb{H}}$ and 
$\mathcal{L}_{1/2}^{\widetilde{\mathtt{H}}_{n}}$ as unions of two independent Poisson point processes:
\begin{displaymath}
\mathcal{L}_{1/2}^{\mathbb{H}}=\mathcal{L}_{\alpha}^{\mathbb{H}}\cup
\mathcal{L}_{\bar{\alpha}}^{\mathbb{H}},\qquad
\mathcal{L}_{1/2}^{\widetilde{\mathtt{H}}_{n}}=\mathcal{L}_{\alpha}^{\widetilde{\mathtt{H}}_{n}}\cup
\mathcal{L}_{\bar{\alpha}}^{\widetilde{\mathtt{H}}_{n}}.
\end{displaymath}

Let $A_{\alpha}$ be the event defined by $\mathcal{F}_{\rm ext}(\mathcal{L}_{1/2}^{\mathbb{H}})(z_{\alpha,0})=
\mathcal{F}_{\rm ext}(\mathcal{L}_{1/2}^{\mathbb{H}})(\tilde{z})$.
Let $A^{+}_{\alpha}$ be the event which holds if one of the below conditions is satisfied:
\begin{itemize}
\item $\mathcal{F}_{\rm ext}(\mathcal{L}_{1/2}^{\mathbb{H}})(z_{\alpha,0})=
\mathcal{F}_{\rm ext}(\mathcal{L}_{1/2}^{\mathbb{H}})(\tilde{z})$,
\item $\mathcal{F}_{\rm ext}(\mathcal{L}_{\bar{\alpha}}^{\mathbb{H}})(\tilde{z})$ surrounds
$Z_{\alpha}$.
\end{itemize}

Figure \ref{CEAplus2} is an illustration of $A^{+}_{\alpha}\setminus A_{\alpha}$.

\begin{figure}[ht]
\begin{center}
  \includegraphics[scale=0.5]{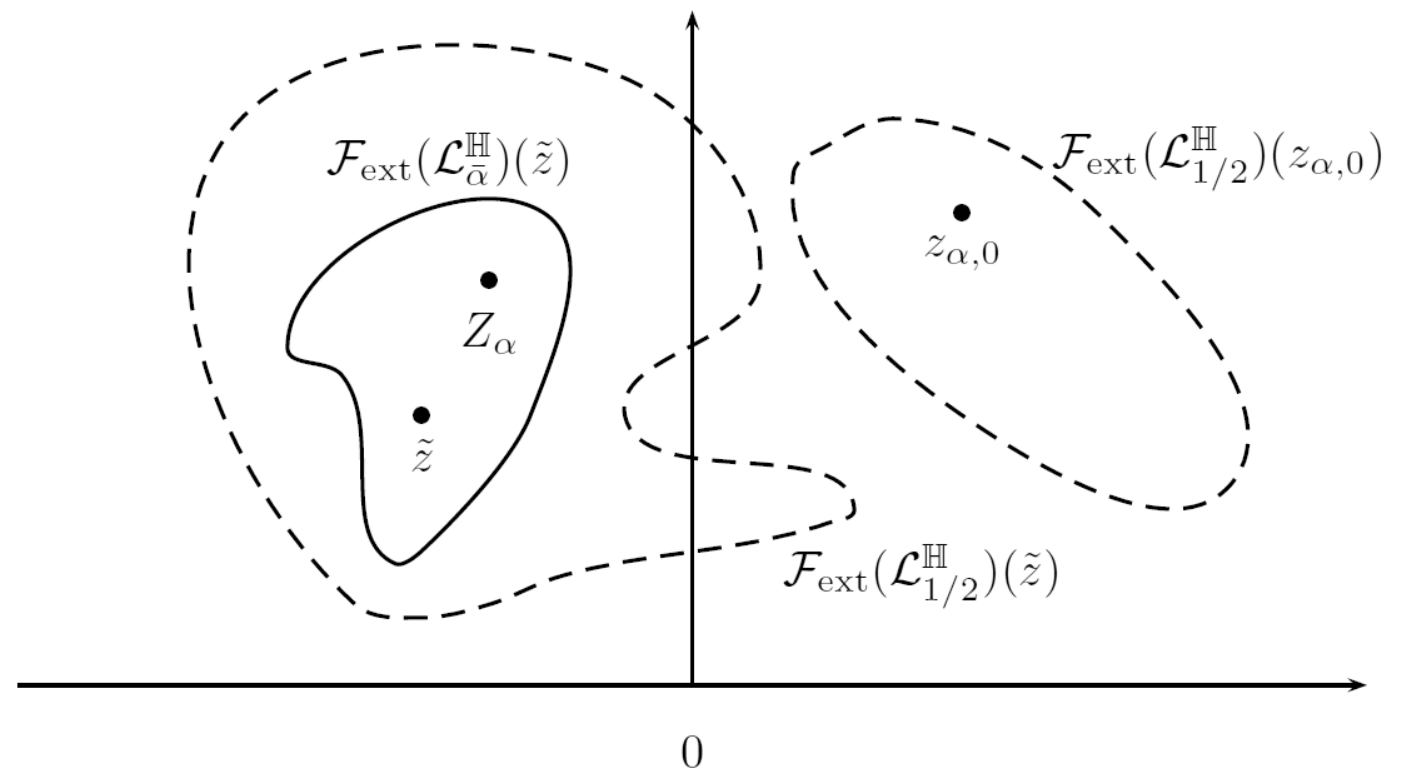}
  \end{center}
  \caption{Illustration of $A^{+}_{\alpha}\setminus A_{\alpha}$.}
  \label{CEAplus2}
\end{figure}

Let us show that
$\mathbb{P}(A^{+}_{\alpha}\setminus A_{\alpha})>0$. Let $E_{4}$ be the event defined by the following four conditions:
\begin{itemize}
\item $d(Z_{\alpha},
\mathcal{F}_{\rm ext}(\mathcal{L}_{\alpha}^{\mathbb{H}})(z_{\alpha,0}))>0$,
\item $\mathcal{F}_{\rm ext}(\mathcal{L}_{\alpha}^{\mathbb{H}})(z_{\alpha,0})\neq
\mathcal{F}_{\rm ext}(\mathcal{L}_{\alpha}^{\mathbb{H}})(\tilde{z})$,
\item $\mathcal{F}_{\rm ext}(\mathcal{L}_{\bar{\alpha}}^{\mathbb{H}})(\tilde{z})$ surrounds
$Z_{\alpha}$,
\item $\mathcal{F}_{\rm ext}^{\bullet}(\mathcal{L}_{\bar{\alpha}}^{\mathbb{H}})(\tilde{z})\cap
\mathcal{F}_{\rm ext}^{\bullet}(\mathcal{L}_{\alpha}^{\mathbb{H}})(z_{\alpha,0})=\emptyset$.
\end{itemize}
It has positive probability because of our choice of $\tilde{z}$ and the independence of
$\mathcal{F}_{\rm ext}(\mathcal{L}_{\bar{\alpha}}^{\mathbb{H}})(\tilde{z})$ from
$(\mathcal{F}_{\rm ext}(\mathcal{L}_{\alpha}^{\mathbb{H}}),Z_{\alpha})$.
Let $\bar{A}_{\alpha}$ be the complement of $A_{\alpha}$. $\bar{A}_{\alpha}$ and $E_{4}$ are independent conditional on
$(\mathcal{F}_{\rm ext}(\mathcal{L}_{\alpha}^{\mathbb{H}}),
\mathcal{F}_{\rm ext}(\mathcal{L}_{\bar{\alpha}}^{\mathbb{H}})(\tilde{z}))$. Thus
\begin{displaymath}
\mathbb{P}(A^{+}_{\alpha}\setminus A_{\alpha})=\mathbb{P}(E_{4},\bar{A}_{\alpha})=
\mathbb{E}[1_{E_{4}}\mathbb{P}(\bar{A}_{\alpha}\vert 
\mathcal{F}_{\rm ext}(\mathcal{L}_{\alpha}^{\mathbb{H}}),
\mathcal{F}_{\rm ext}(\mathcal{L}_{\bar{\alpha}}^{\mathbb{H}})(\tilde{z}))].
\end{displaymath}
According to Lemma \ref{LemPositivity}, 
$\mathbb{P}(\bar{A}_{\alpha}\vert 
\mathcal{F}_{\rm ext}(\mathcal{L}_{\alpha}^{\mathbb{H}}),
\mathcal{F}_{\rm ext}(\mathcal{L}_{\bar{\alpha}}^{\mathbb{H}})(\tilde{z}))$ is a.s.\ positive on the event $E_{4}$.
It follows that $\mathbb{P}(A^{+}_{\alpha}\setminus A_{\alpha})>0$.

Let $A^{n}_{\alpha}$ and $A^{n,+}_{\alpha}$ be the events defined similarly to $A_{\alpha}$ and $A^{+}_{\alpha}$ where
the contours $\mathcal{F}_{\rm ext}(\mathcal{L}_{1/2}^{\mathbb{H}})(z_{\alpha,0})$, 
$\mathcal{F}_{\rm ext}(\mathcal{L}_{1/2}^{\mathbb{H}})(\tilde{z})$ and
$\mathcal{F}_{\rm ext}(\mathcal{L}_{\bar{\alpha}}^{\mathbb{H}})(\tilde{z})$ are replaced by
\\$\mathcal{F}_{\rm ext}(\mathcal{L}_{1/2}^{\widetilde{\mathtt{H}}_{n}})(z_{\alpha,0})$, 
$\mathcal{F}_{\rm ext}(\mathcal{L}_{1/2}^{\widetilde{\mathtt{H}}_{n}})(\tilde{z})$ and
$\mathcal{F}_{\rm ext}(\mathcal{L}_{\bar{\alpha}}^{\widetilde{\mathtt{H}}_{n}})(\tilde{z})$ respectively and
$Z_{\alpha}$ is replaced by $Z_{\alpha,n}$. Since $Z_{\alpha,n}$ is on the contour
$\mathcal{F}_{\rm ext}(\mathcal{L}_{\alpha}^{\widetilde{\mathtt{H}}_{n}})(z_{\alpha,0})$ we have the equality
$A^{n,+}_{\alpha}=A^{n}_{\alpha}$. From Theorem \ref{ThmConvOneHalf} follows that
\begin{displaymath}
\lim_{n\rightarrow +\infty}\mathbb{P}(A^{n}_{\alpha})=\mathbb{P}(A_{\alpha}).
\end{displaymath}
On the other hand
\begin{displaymath}
\liminf_{n_{\alpha,0}\rightarrow +\infty}\mathbb{P}(A^{n_{\alpha,0},+}_{\alpha})\geq\mathbb{P}(A^{+}_{\alpha})>
\mathbb{P}(A_{\alpha}),
\end{displaymath}
which is a contradiction. It follows that $\mathcal{F}_{\rm ext}(\mathcal{L}_{\alpha}^{\widetilde{\mathtt{H}}_{n}})$ converges in law to $\mathcal{F}_{\rm ext}(\mathcal{L}_{\alpha}^{\mathbb{H}})$.
\end{proof}

\section*{Acknowledgements}

This research was supported by Université Paris-Sud, Orsay.

\bigskip

The author thanks Wendelin Werner for explaining the theory of restriction measures and pointing out the $1/2$ factor in the relation between the loop-soup intensity parameter and the central charge.

\end{document}